\documentclass[10pt,reqno]{amsart}

\usepackage{amssymb}
\usepackage{graphicx}
\usepackage{amsfonts}
\usepackage{amsmath}
\usepackage{amsthm}
\usepackage{fancyhdr}
\usepackage{indentfirst}
\usepackage{hyperref}
\usepackage{comment}
\usepackage{color}
\usepackage{subcaption}
\usepackage{epsfig}
\usepackage{pst-grad} 
\usepackage{pst-plot} 
\usepackage[space]{grffile} 
\usepackage{etoolbox} 
\usepackage{mathrsfs} 
\usepackage{mathtools} 

\newtheorem{theorem}{Theorem}[section]
\newtheorem{lemma}[theorem]{Lemma}

\theoremstyle{definition}

\newtheorem{proposition}[theorem]{Proposition}

\theoremstyle{remark}
\newtheorem{remark}[theorem]{Remark}

\numberwithin{equation}{section}

\mathtoolsset{showonlyrefs=true}

\newcommand{\ep}{\varepsilon}

\newcommand{\R}{\mathbb{R}}
\newcommand{\N}{\mathbb{N}}

\newcommand{\Sph}{\mathbb{S}}

\newcommand{\Hcal}{\mathcal{H}}

\newcommand{\Gr}{\mathbf{G}}

\newcommand{\bV}{\mathbf{V}}

\newcommand{\bIV}{\mathbf{IV}}

\newcommand{\lc}{%
  \,\raisebox{-.127ex}{\reflectbox{\rotatebox[origin=br]{-90}{$\lnot$}}}\,%
}

\newcommand{\mres}{\mathbin{\vrule height 1.6ex depth 0pt width
0.13ex\vrule height 0.13ex depth 0pt width 1.3ex}}

\newcommand{\spt}{\operatorname{spt}}
\newcommand{\dist}{\operatorname{dist}}
\newcommand{\Div}{\operatorname{div}}

\newcommand{\inj}{\operatorname{inj}}
\newcommand{\Hess}{\operatorname{Hess}}

\begin{document}

\title[Boundary behavior of limit-Interfaces]{Boundary behavior of limit-Interfaces for the Allen-Cahn equation on Riemannian manifolds with Neumann boundary condition}

\author[Martin Li]{Martin Man-chun Li}
\address{Department of Mathematics, The Chinese University of Hong Kong, Shatin, N.T., Hong Kong}
\email{martinli@math.cuhk.edu.hk}

\author[Davide Parise]{Davide Parise}
\address{Department of Mathematics, University of California San Diego, 9500 Gilman Drive \#0112, La Jolla, CA 92093-0112, USA}
\email{dparise@ucsd.edu}

\author[Lorenzo Sarnataro]{Lorenzo Sarnataro}
\address{Department of Mathematics, Princeton University, Fine Hall, Washington Road, Princeton, NJ 08544-1000 USA}
\email{lorenzos@princeton.edu}

\begin{abstract}
We study the boundary behavior of any limit-interface arising from a sequence of general critical points of the Allen-Cahn energy functionals on a smooth bounded domain. Given any such sequence with uniform energy bounds, we prove that the limit-interface is a free boundary varifold which is integer rectifiable up to the boundary. This extends earlier work of Hutchinson and Tonegawa on the interior regularity of such limit-interface. A key novelty in our result is that no convexity assumption of the boundary is required and it is valid even when the limit-interface clusters near the boundary. Moreover, our arguments are local and thus works in the Riemannian setting. This work provides the first step towards the regularity theory for the Allen-Cahn min-max theory for free boundary minimal hypersurfaces, which was developed in the Almgren-Pitts setting by the first-named author and Zhou.
\end{abstract}

\maketitle

\section{Introduction}
\label{S:intro}

Let $\Omega \subset \R^n$, $n \geq 2$, be a smooth bounded domain. For each small parameter $\ep >0$, consider the \emph{Allen-Cahn energy functional} arising from the van der Waals-Cahn-Hilliard theory of phase transition \cite{Cahn-Hilliard58}:
\begin{equation}
\label{E:CH-energy}
E_\ep(u):=\int_\Omega \frac{\ep|\nabla u|^2}{2} + \frac{W(u)}{\ep} 
\end{equation}
where $u: \Omega \to \R$ belongs to the Sobolev space $H^1(\Omega)=\{u \in L^2(\Omega):\nabla u \in L^2(\Omega)\}$. Here, $W: \R \to \R$ is a nonnegative double-well potential with strict minima at $\pm 1$ with $W(\pm 1)=0$. If $u \in H^1(\Omega)$ is a critical point of $E_\ep$, by standard elliptic theory $u \in C^3(\overline{\Omega})$ and satisfies the Euler-Lagrange equation:
\begin{equation}
\label{E:Allen-Cahn}
\left\{      
\begin{array}{cl}
-\ep^2 \Delta u + W'(u)=0 & \text{in $\Omega$}\\
\partial u / \partial \nu=0 & \text{on $\partial \Omega$}
\end{array}
\right. 
\end{equation}
where $\nu$ is the outward (with respect to $\Omega$) unit normal of $\partial \Omega$. In the literature, (\ref{E:Allen-Cahn}) is often known as the \emph{Allen-Cahn equation}. This equation and its parabolic counterpart (see e.g. \cite{Rubinstein1989}) account for important models in phase transitions and evolution of antiphase boundaries \cite{Allen-Cahn79}.

There is an extensive literature on the semi-linear elliptic equation \eqref{E:Allen-Cahn}. Much of the development was centered around a celebrated conjecture of De Giorgi concerning entire (i.e. $\Omega = \mathbb{R}^n$) solutions. We refer the interested reader to an excellent recent survey \cite{Chan2018} about this conjecture and the recent developments.

The Allen-Cahn equation is closely related to the theory of minimal hypersurfaces. This connection is particularly important when one considers the asymptotic behavior of solutions to the Allen-Cahn equation as $\ep \to 0^+$. Heristically, the level sets of solutions $u_\ep$ to the Allen-Cahn equation will converge in some sense to a sharp limit interface which is an embedded minimal hypersurface (which may possess singularity). For minimizing solutions, this was made precise using De Giorgi's $\Gamma$-convergence theory by Modica \cite{Modica1987a} and Kohn-Sternberg \cite{Kohn1989}. They further used this idea to show the existence of stable solutions for \eqref{E:Allen-Cahn} in two-dimensional non-convex domains such as a dumbbell. The existence of unstable solutions in two-dimensions was later studied by Kowalczyk \cite{Kowalczyk2005}. Sternberg and Zumbrun in \cite{Sternberg1998} proved the connectivity of the phase boundaries in strictly convex domains. 

The interior regularity of \eqref{E:Allen-Cahn} has been treated in a series of important papers by Padilla-Tonegawa \cite{Padilla1998}, Hutchinson-Tonegawa \cite{Hutchinson2000}, Tonegawa \cite{Tonegawa2005a} and Tonegawa-Wickramasekera \cite{Tonegawa2012}. We remark that these results are foundational in a recent beautiful work of Guaraco \cite{Guaraco2018} on the min-max theory for minimal hypersurfaces in closed Riemannian manifolds using the Allen-Cahn approach. This approach is especially successful when the ambient space has dimension two and three, as seen in the beautiful work of Chodosh and Mantoulidis \cite{Mantoulidis2021,Chodosh2023,Chodosh2020}. 

Although we have by now a rather satisfactory interior theory, much less is known about the boundary behavior of \eqref{E:Allen-Cahn} as $\ep \to 0^+$. One major obstacle is to understand the phenomenon of boundary concentration. When $\Omega$ is the unit ball, Malchiodi, Ni and Wei \cite{Malchiodi2007a} constructed radially symmetric solutions $u_\ep$ to \eqref{E:Allen-Cahn} whose interfaces clustered towards $\partial \Omega$. Malchiodi and Wei \cite{Malchiodi2007} removed the symmetry assumption and showed that such boundary clustering can occur for any mean-convex domain $\Omega \subset \mathbb{R}^n$. These results show that the limit-interface can stick to the boundary even when $\Omega$ is \emph{convex}.

Our goal is to fill in this gap in the literature and, in subsequent works, to generalize the Allen-Cahn min-max theory to the free boundary setting. More precisely, in this paper we study the boundary behavior of the limit-interface arising from general critical points of the Allen-Cahn energy functional $E_\ep$ in \eqref{E:CH-energy}, i.e. solutions satisfying \eqref{E:Allen-Cahn}, subject only to a uniform energy bound. In particular, we do not assume that $u_\ep$ is \emph{minimizing} or \emph{stable}. We show that the limit-interface is \emph{integer rectifiable} up to the boundary (i.e. when regarded as an $(n-1)$-varifold \emph{in $\R^n$}). Moreover, the limit varifold is \emph{stationary with free boundary}, hence a critical point of the area functional with respect to variations preserving the boundary $\partial \Omega$ as a set (we note that a different Dirichlet boundary value problem, related to Plateau's problem, was recently studied by Guaraco and Lynch \cite{Guaraco})  This provides the optimal regularity of the limit-interface up-to-the-boundary, extending the interior regularity result of Hutchinson-Tonegawa \cite{Hutchinson2000} to the free boundary setting.



We now explain in more details what the key difficulties in extending the interior regularity theory to the boundary are and how we overcome these new difficulties. One of the main ingredients in the regularity theory for general critical points is an almost monotonicity formula. The interior monotonicity-type formula was obtained in \cite[Proposition 3.4]{Hutchinson2000}. Using a maximum principle argument and a reflection process as in \cite{Grueter1986}, Tonegawa \cite[Theorem 3.7]{Tonegawa2003} proved a \emph{global} monotonicity formula up to the boundary provided that $\Omega$ is \emph{convex}. One of our main results is a boundary monotonicity-type formula, Proposition \ref{P:almost-monotonicity}, which differs from that of \cite{Tonegawa2003} in two respects. First, our result is local and does not require the convexity of $\partial \Omega$. Without the convexity assumption, we can only established a weaker pointwise upper bound (c.f. Proposition \ref{P:xi-bound}) for the discrepancy function compared to the interior case \cite[Proposition 3.3]{Hutchinson2000} and the global case in convex domains \cite[Proposition 3.1]{Tonegawa2003}. Nonetheless, our weaker bound, when combined with suitable integral estimates (c.f. Proposition \ref{P:xi-integral-bound}), is sufficient to establish the desired local monotonicity formula up to the boundary. We would like to point out that a similar upper bound for the discrepancy function in the interior was also observed with variable chemical potential lying in certain Sobolev spaces \cite{Tonegawa2005}. Second, the monotonicity formula in \cite{Tonegawa2003} contains terms integrated on balls reflected across the boundary as in \cite{Grueter1986}, which sometimes make it more difficult to apply. Our monotonicity formula does not involve any reflection and thus can be more readily used in applications. Note that a similar monotonicity formula for free boundary stationary varifolds was obtained by the first-named author with Guang and Zhou in \cite{Guang2020}. It would be possible---though we chose not to pursue it here---to use a reflection technique similar to the one adopted in \cite{Grueter1986} and \cite{Tonegawa2003}, and obtain a monotonicity formula involving the energy density on balls reflected across the boundary, resembling that of \cite{Tonegawa2003}.

Without considering reflected balls, another technical difficulty is that our monotonicity type formula only applies to balls that are either disjoint from the boundary or centered exactly at the boundary. This makes the arguments much more delicate when applying the monotonicity formula. For example, in contrast to the interior theory \cite{Hutchinson2000}, we can only establish uniform density ratio upper bound up to the boundary but not the lower bound (c.f. Proposition \ref{P:mu-density-bound}). This is related to the boundary clustering phenomena which is a major obstacle in our setting. Nonetheless, we can still prove the vanishing of the discrepancy measure (c.f. Theorem \ref{T:xi-vanish}) which allows us to obtain control on the first variation of the limit interface. Given the weaker control on the discrepancy function and the potential boundary clustering phenomena, we carefully modify the arguments in \cite{Hutchinson2000} to establish the up-to-the-boundary integer-rectifiability of the limit varifold.


The paper is organized as follows. In Section \ref{S:prelim}, we introduce the notations and some preliminary results that will be used throughout the rest of the paper. We then state our assumptions and the main result of this paper (Theorem \ref{T:main}) in Section \ref{S:main}. In Section \ref{S:monotonicity}, we derive our key monotonicity-type formula and several of its consequences. Much of the work is devoted to proving pointwise and integral estimates on the discrepancy function. Using our monotonicity formula, we then derive at the diffuse level (i.e. $\ep>0$) uniform bounds on the energy density ratio for balls that are either centered at a boundary point or completely disjoint from the boundary. In Section \ref{S:rectifiable}, we establish bounds on the density ratio for the limit measure and prove the vanishing of discrepancy measure, which implies the rectifiability of the limit interface. The last section, Section \ref{S:integrality}, shows further that the limit interface is in fact an integral varifold up to the boundary.

\vspace{1ex}

\noindent \textbf{Acknowledgement.} 
M. Li would like to express his deep gratitude to Prof. Richard Schoen for his constant support and encouragement. He would also like to thank Xin Zhou and Qi Ding for their interest of this work. The author is substantially supported by research grants from the Excellent Young Scientist Fund from the National Natural Science Foundation of China [Project No.: 12022120] and the Research Grants Council of the Hong Kong Special Administrative Region, China [Project No.: CUHK 14301319 and CUHK 14304121].
D. Parise would like to thank L. Spolaor for helpful discussions. L. Sarnataro is grateful to Fernando Codá Marques for his guidance. He would also like to thank Douglas Stryker, Otis Chodosh, and Daniel Stern for helpful discussions.

\section{Notations and preliminaries}
\label{S:prelim}

In this section, we introduce the notations and preliminary concepts that will be used throughout the rest of this paper.

\subsection{Ambient space}
\label{SS:ambient}

Let $(\widetilde{M},g)$ be a smooth complete $n$-dimensional Riemannian manifold \emph{without boundary}. We use $\dist_g$ and $\nabla$ to denote respectively the intrinsic distance and Levi-Civita connection on $\widetilde{M}$ induced by $g$. For any $x \in \widetilde{M}$, we denote by $\exp_x$ the exponential map and $\textrm{inj}_g(x)$ the injectivity radius of $(\widetilde{M},g)$ at $x$. We use $B_r(x)$ to denote the open geodesic ball of radius $r$ centered at $x \in \widetilde{M}$. For a given subset $A \subset \widetilde{M}$, we denote by $\mathbf{1}_A$ its indicator function; by $\overline{A}$ and $A^\circ$ respectively the closure and interior of $A$ with respect to the topology on $(\widetilde{M},g)$. Moreover, we set $\inj_g(A):=\inf_{x \in A} \inj_g(x)$ and $U_r(A):=\cup_{x \in A} B_r(x)$. For any linear transformations $A,B :T_x \widetilde{M} \to T_x \widetilde{M}$, we define the Frobenius (or Hilbert-Schmidt) inner product and norm by
\[ \langle A, B \rangle_g:= \sum_{i,j=1}^n A_{ij} B_{ij} \quad \text{ and } \quad |A|_g:=\sqrt{\langle A, A \rangle}\]
where $A_{ij}, B_{ij}$ are matrix components of $A, B$ under an orthonormal basis of $T_x \widetilde{M}$ with respect to $g$. Moreover, $\|A\|_g$ denotes the operator norm. We shall sometimes use $\cdot$ to denote the inner product $g$ on the tangent spaces of $\widetilde{M}$ or omit the subscript $g$ when there is no ambiguity. Vectors and co-vectors are identified via the metric $g$ wherever necessary.

Throughout this paper (see Section \ref{SS:assumption}), we assume $(\widetilde{M},g)$ has positive injectivity radius $\inj_g(\widetilde{M})\geq i_0>0$ and bounded sectional curvature, i.e. $|\sec_{(\widetilde{M},g)}| \leq k_0$, for some constant $k_0 \geq 0$. The classical Hessian comparison theorem implies (see \cite[Lemma 7.1]{Colding2011}) that the distance function $r(\cdot)=\dist_g(\cdot,x)$ from any $x \in \widetilde{M}$ satisfies
\begin{equation}
\label{E:Hessian-comparison}
\left| \Hess r (X,X) - \frac{1}{r} |X- (X \cdot \nabla r) \nabla r|^2 \right| \leq \sqrt{k_0}
\end{equation} 
provided that $|X| =1$ and $r < \min\{i_0,\frac{1}{\sqrt{k_0}}\}$. Taking the trace of \eqref{E:Hessian-comparison} along directions orthogonal to $\nabla r$ yields
\begin{equation}
\label{E:Laplacian-comparison}
\left| \Delta r - \frac{n-1}{r} \right| \leq (n-1) \sqrt{k_0}.
\end{equation}
Moreover, the Bochner formula \cite[Proposition 1.47]{Colding2011} implies that for any smooth function $f$ on subsets of $\widetilde{M}$,
\begin{equation}
\label{E:Bochner}
\left| \frac{1}{2} \Delta |\nabla f|^2 - |\Hess f|^2- \nabla f \cdot \nabla (\Delta f) \right| \leq (n-1) k_0 |\nabla f|^2.
\end{equation}
On the other hand, by Rauch comparison theorem, there exists a constant $r_0>0$ depending only on $n$ and $k_0$ such that for all $r<\min\{r_0,i_0\}$, we have
\begin{equation}
\label{E:ball-comparison}
\frac{1}{2} \omega_n r^n \leq  \textrm{Vol}(B_r(x)) \leq 2 \omega_n r^n,
\end{equation}
\begin{equation}
\label{E:sphere-comparison}
\frac{1}{2} n\omega_n r^{n-1} \leq  \textrm{Area}(\partial B_r(x)) \leq 2 n \omega_n r^{n-1}.
\end{equation}
Here, $\omega_k$ is the volume of the $k$-dimensional unit ball in $\mathbb{R}^k$ with respect to the Euclidean metric. 

Let $\Gr_x(n,k)$ be the space of unoriented $k$-dimensional subspaces of $T_x \widetilde{M}$. The orthogonal complement of $S \in \Gr_x(n,k)$ is denoted by $S^\perp \in \Gr_x(n,n-k)$. For $a \in T_x \widetilde{M}$, $a \otimes a :T_x \widetilde{M} \to T_x \widetilde{M} $ is the matrix with components $a_i a_j$ with respect to an orthonormal basis. For $S \in \Gr_x(n,k)$, we identify $S$ with the corresponding orthogonal projection of $T_x \widetilde{M}$ onto $S$. In the case of $k=n-1$, we also identify $S \in \Gr_x(n,n-1)$ with the unit vector $\pm N \in T_x \widetilde{M}$ which is perpendicular to $S$. Note that we may express the relation by $S=I- N \otimes N$ where $I$ is the identity map on $T_x \widetilde{M}$. The correspondence is a homeomorphism with respect to the naturally endowed topologies on $\Gr_x(n,n-1)$ and $\Sph^{n-1}/\{\pm 1\}$.

\subsection{Boundary geometry}
\label{SS:extrinsic-geom}

Let $\Omega \subset \widetilde{M}$ be an open and connected domain with smooth non-empty boundary $\partial \Omega$, whose outward unit normal is denoted by $\nu$. As in \cite{Brendle2016}, we define the local interior ball curvature at $x \in \partial \Omega$ by 
\[ \overline{Z}_\Omega (x):= \sup_{y \in \partial \Omega, 0 < \dist_g(x,y) \leq \frac{1}{2} \inj_g(\widetilde{M})} \left( - \frac{ 2 \exp^{-1}_x(y) \cdot \nu(x)}{\dist_g(x,y)^2} \right)\]
where $\exp^{-1}_x(y)$ is regarded as a vector in $T_x \widetilde{M}$ emanating from the origin. One can similarly define the local exterior ball curvature $\underline{Z}_\Omega(x)$ by considering the infimum instead. When $(\widetilde{M},g)$ is $\mathbb{R}^n$, this agrees with the notion of (global) interior and exterior ball curvature as given in \cite{Andrews2013} (c.f. \cite{Hirsch2020}). By considering $y \to x$, it is clear that $\underline{Z}_\Omega(x) \leq \lambda_1(x) \leq \cdots \leq \lambda_{n-1}(x) \leq \overline{Z}_\Omega (x)$ where $\lambda_i(x)$ are the principal curvatures of $\partial \Omega$ at $x$ with respect to the normal $\nu$. Note that $\overline{Z}_\Omega=\underline{Z}_\Omega=0$ when $\Omega$ is a half space in $\mathbb{R}^n$.

We will always assume that $\Omega$ has bounded ball curvatures, i.e. $\max\{|\overline{Z}_\Omega|, |\underline{Z}_\Omega|\} \leq \kappa_0$ for some constant $\kappa_0 \geq 0$. Therefore, the signed distance function $d(\cdot)=\dist_g(\cdot, \partial \Omega)$, with $d>0$ in $\Omega$, is smooth in the tubular neighborhood $U_{\frac{1}{2\kappa_0}}(\partial \Omega)$ satisfying the following \footnote{See \cite[Appendix 14.6]{Gilbarg2001} for the calculations in $\mathbb{R}^n$. The Riemannian case follows from standard Jacobi field estimates.} :
\begin{equation}
\label{E:dist-estimates}
|\nabla d| =1, \quad \Hess d(\nabla d, \cdot)=0 \quad \text{ and } \quad \| \Hess d\| \leq C_1
\end{equation}
where $C_1>0$ is a constant depending only on $n$, $k_0$ and $\kappa_0$. Moreover, for any $x \in \partial \Omega$, the radial distance function $r(\cdot)=\dist_g(\cdot,x)$ restricted to $\partial \Omega$ satisfies
\begin{equation}
\label{E:allard-estimate}
|r \nabla r \cdot \nu | \leq \frac{\kappa_0}{2} r^2
\end{equation}
provided that $r<\frac{1}{2} \inj_g(\widetilde{M})$. This follows from the definition and the observation that $\exp^{-1}_x y$ is the parallel transport of $r \nabla r$ at $y$ back to $x$ along a radial geodesic. It is also easily seen from the definition that there exists a constant $r_1>0$ depending only on $n$, $k_0$ and $\kappa_0$ such that for any $x \in \partial \Omega$, $r<\min\{r_1,i_0\}$, we have 
\begin{equation}
\label{E:half-ball-comparison}
\frac{1}{3} \omega_n r^n \leq  \textrm{Vol}(B_r(x) \cap \Omega) \leq \frac{2}{3} \omega_n r^n,
\end{equation}
\begin{equation}
\label{E:half-sphere-comparison}
\frac{1}{3} n \omega_n r^{n-1} \leq  \textrm{Area}(\partial B_r(x) \cap \Omega) \leq  \frac{2}{3} n \omega_n r^{n-1}.
\end{equation}

\subsection{Functions, vector fields and measures}
\label{SS:notation}

Let $U \subset \widetilde{M}$ be an open subset. We use $C_c(U)$ to denote the space of real-valued continuous functions on $U$ with compact support; $L^p(U)$ to denote the space of functions on $U$ which are in $L^p$ with respect to the volume measure $dv_g$ on $(\widetilde{M},g)$. For any subset $A \subset \widetilde{M}$, we let $C^k(A)$ be the space of all functions defined on $A$ which can be extended to a $C^k$ function on an open subset $U \subset \widetilde{M}$ containing $A$. We will use $f^+:=\max\{f,0\}$ and $f^-:=\max\{-f,0\}$ to denote respectively the positive and negative part of a real-valued function $f$. 

Let $\mathfrak{X}_c(U)$ be the space of smooth, compactly supported vector fields on $U$; and $\mathfrak{X}^{tan}_c(U)$ be the subspace consisting of all $X \in \mathfrak{X}_c(U)$ such that $X(x) \in T_x \partial \Omega$ for all $x \in \partial \Omega \cap U$. Note that $\mathfrak{X}^{tan}_c(U)=\mathfrak{X}_c(U)$ if $\partial \Omega \cap U = \emptyset$. A function $u \in L^1(U)$ is said be in $BV(U)$ (see \cite{Evans2015}, and also \cite{Miranda2007} in the Riemannian setting) if 
\[ \sup \left\{ \int_U u \Div_g X \; dv_g \; : \; X \in \mathfrak{X}_c(U), \; |X| \leq 1\right\} < +\infty \]
where $\Div_g$ is the divergence operator with respect to $g$. For any $u \in BV(U)$, we use $Du$ to denote its distributional derivative which is a vector-valued Radon measure on $U$. The variation measure is then denoted by $\|Du\|$. A subset $A \subset U$ is said to have finite perimeter in $U$ if $\mathbf{1}_A \in BV(U)$, and we call $\|\partial A\|:=\|D \mathbf{1}_A\|$ the perimeter measure and $\|\partial A\|(U):=\|D \mathbf{1}_A\|(U)$ the perimeter of $A$ in $U$. We use $\partial^* A$ to denote the reduced boundary and $\nu_{\partial^* A}$ the generalized outward normal (c.f. \cite[Section 5.7]{Evans2015}). These concepts can be defined locally, as in \cite[Definition 5.2]{Evans2015}.

For any $0 \leq k \leq n$, the $k$-dimensional Hausdorff measure on $(\widetilde{M},g)$ is denoted by $\mathcal{H}^k$. Note that $\mathcal{H}^n$ agrees with the volume measure $dv_g$ on $(\widetilde{M},g)$. Given a subset $A \subset U$ and a measure $\mu$ on $U$, the restriction of $\mu$ to $A$ is denoted by $\mu \mres A$. For any Radon measure $\mu$ on $U$ and $\phi \in C_c(U)$ we often write $\mu(\phi)$ for $\int \phi d\mu$. A sequence $\{\mu_i\}_{i \in \mathbb{N}}$ of Radon measures on $U$ is said to converge (in the weak-$*$ topology) to a Radon measure $\mu$ on $U$ if $\mu_i(\phi) \to \mu(\phi)$ for all $\phi \in C_c(U)$. We will denote this weak convergence by $\mu_i \to \mu$. We omit $d\mu$ if $\mu$ is the volume measure $dv_g$. We write $\spt \mu$ for the support of $\mu$, which is the smallest relatively closed subset $V \subset U$ such that $\mu(U\setminus V)=0$. Thus $x \in \spt \mu$ if and only if $\mu(B_r(x)) >0$ for all $r>0$. For any $x \in U$, we denote the $k$-dimensional density of $\mu$ at $x$ as
\begin{equation}
\label{E:density}
\Theta^{k}(\mu,x):=\lim_{r \to 0} \frac{\mu(B_r(x))}{\omega_k r^k} 
\end{equation} 
provided that the limit exists. Clearly $\Theta^k(\mu,x)=0$ for all $x \notin \spt \mu$. Therefore, we often regard $\Theta^k(\mu,\cdot)$ as a function defined on $\spt \mu$.

\subsection{Varifolds}
\label{SS:varifolds}

We recall some definitions from geometric measure theory and refer the readers to the standard references \cite{Allard1972,Simon1983} for more details. A nice description of varifolds in Riemannian manifolds without using isometric embedding can be found in \cite{Scharrer2022}. We shall only consider codimension one varifolds in this paper. 

For any open set $U \subset \widetilde{M}$, let $G(U):=\cup_{p \in U} \Gr_p(n,n-1)$. A general varifold in $U$ is a Radon measure on $G(U)$. We denote the set of all general varifolds in $U$ by $\bV_{n-1}(U)$, equipped with the weak topology. For $V \in \bV_{n-1}(U)$, let $\|V\|$ be the weight of $V$, which is the Radon measure on $U$ defined by
\[ \|V\| (\phi) :=\int_{G(U)} \phi(x) \; dV(x,S), \qquad \forall \phi \in C_c(U).\]
We say $V \in \bV_{n-1}(U)$ is rectifiable if there exists an $\Hcal^{n-1}$-measurable, countably $(n-1)$-rectifiable subset $\Sigma \subset U$ and a locally $\Hcal^{n-1}$-integrable function $\Theta$ defined on $\Sigma$ such that 
\[ V(\phi)=\int_\Sigma \phi(x,T_x \Sigma) \; \Theta(x)  d\Hcal^{n-1}(x) \qquad \forall \phi \in C_c(G(U)).\]
Here $T_x \Sigma \in \Gr_x(n,n-1)$ is the unique approximate tangent space of $\Sigma$ at $x$ which exists $\Hcal^{n-1}$-a.e. on $\Sigma$. Note that $\Theta^{n-1}(\|V\|,x)=\Theta(x)$ for any such $x$. A rectifiable varifold $V$ is uniquely determined by its weight $\|V\|=\Theta \Hcal^{n-1} \mres \Sigma$ through the formula above. For this reason, we naturally say a Radon measure $\mu$ on $U$ is rectifiable when one can associate a rectifiable varifold $V$ such that $\|V\|=\mu$. If $\Theta(x) \in \N$ for $\Hcal^{n-1}$-a.e. $x \in \Sigma$, we say $V$ is integer rectifiable or an integral varifold. The set of all integral varifolds in $U$ is denoted by $\bIV_{n-1}(U)$. We say that $V$ is a unit density varifold if $\Theta(x)=1$ for $\|V\|$-a.e. $x$.

For $V \in \bV_{n-1}(U)$ let $\delta V$ be the first variation of $V$, namely,
\begin{equation}
\label{E:1st-variation}
\delta V(X):=\int_{G(U)} \langle \nabla X(x),S \rangle_g \; dV(x,S) \qquad \forall X \in \mathfrak{X}_c(U).
\end{equation}
We also denote the total variation of $\delta V$ by $\|\delta V\|$. If $\|\delta V\|$ is a Radon measure on $U$ (i.e. $V$ has locally bounded first variation on $U$), we may apply the Radon-Nikodym theorem to $\delta V$ with respect to $\|V\|$. Writing the singular part of $\|\delta V\|$ with respect to $\|V\|$ as $\sigma_V$, we have a $\|V\|$-measurable vector field $H$, a $\|\delta V\|$-measurable vector field $\eta$ with $|\eta|=1$ $\| \delta V\|$-a.e., and a Borel set $Z \subset U$ such that $\|V\|(Z)=0$ and
\begin{equation}
\label{E:Radon-Nikodym} 
\delta V(X)=-\int_U X(x) \cdot H(x) \; d\|V\|(x) +\int_Z X(x) \cdot \eta(x) \; d\sigma_V(x)
\end{equation}
for all $X \in \mathfrak{X}_c(U)$. We call $H$ the generalized mean curvature vector, $\eta$ the generalized outward co-normal, $\sigma_V$ the generalized boundary measure, and $Z$ the generalized boundary of $V$.

To study the boundary behavior, we shall also need the notion of \emph{free-boundary varifolds} introduced in \cite[Section 3]{Edelen2020} (see also \cite[Definition 2.4]{Masi2020}). Given a smoothly embedded hypersurface $S \subset U$ oriented by the normal $\nu_S$, a varifold $V \in \bV_{n-1}(U)$ is said to have free boundary in $S$ if there exists a $\|V\|$-measurable vector field $H^T$ which is locally integrable with respect to $\|V\|$ and $H^T(x) \cdot \nu_S (x) =0$ for $\|V\|$-a.e. $x \in S$, satisfying
\begin{equation}
\label{E:free-boundary-varifold}
\delta V(X)= - \int_U X(x) \cdot H^T(x) \; d\|V\|(x)
\end{equation}
for all $X \in \mathfrak{X}_c(U)$ with $X \cdot \nu_S=0$ on $S$. Note that we use the notation $H^T$ to distinguish it from the generalized mean curvature vector $H$ defined in \eqref{E:Radon-Nikodym}. In fact, if $H$ exists, then $H^T$ is the tangential part of $H$ along $S$ (see \cite[Definition 3.0.3]{Edelen2020}). It was shown in \cite[Corollary 4.7]{Masi2020} \footnote{Although the statement was presented for varifolds in Euclidean space, the arguments there are local and hence the proof carries over to the Riemannian setting.} that any such free-boundary varifold with $H^T \in L^1_{loc}$ with respect to $\|V\|$ has locally bounded first variation in $U$. By \cite[Proposition 3.2]{Edelen2020}, an integral varifold $V \in \bIV_{n-1}(U)$ has free boundary in $S$ if and only if $V$ has locally bounded first variation in $U$, and the generalized boundary measure $\sigma_V$ is supported in $S$, and $\eta=\nu_S$ at $\sigma_V$-a.e. $x \in Z$. For our case at hand, we will mostly take $S=\partial \Omega$.

\subsection{Double-well potential}
\label{SS:double-well}

We fix a double-well potential $W:\mathbb{R} \to \mathbb{R}$ which is a non-negative smooth function satisfying the following:
\begin{itemize}
\item $W(\pm1)=W'(\pm1)=0$,
\item $W$ has a unique local non-degenerate maxima at $\gamma \in (-1,1)$,
\item there exist some $\alpha \in (0,1)$ and $\kappa \in (0,1)$ such that $W''(t) \geq \kappa$ for all $|t| \geq \alpha$.
\end{itemize}
We also define the energy constant:
\begin{equation}
\label{E:sigma}
h_0:=\int_{-1}^1 \sqrt{2 W(s)} \; ds.
\end{equation}

\begin{remark}
A common choice of the double-well potential is given by $W(s)=\frac{1}{4}(1-s^2)^2$ with the unique local maxima at $\gamma=0$. Note that our definition of the constant $h_0$ differs from \cite{Hutchinson2000,Guaraco2018} (denoted by $\sigma$ there) by a factor of $2$ but agrees with that in \cite{Kagaya2018}. Some of our results established in this paper actually hold for more general potentials with multiple wells, such as those treated in \cite{Tonegawa2003}.  
\end{remark}

\subsection{Constants and cutoffs}
\label{SS:cutoff}

We shall use the convention that $C(\alpha,\beta,\gamma,\dots)$ denotes a positive constant (which may change from line to line) that depends only on the parameters $\alpha, \beta,\gamma,\dots$.
Throughout this paper, we fix a smooth cutoff function $\chi:[0,\infty) \to [0,\infty)$ such that
\begin{itemize}
\item $\chi(s)=s$ for $s \in [0,1]$,
\item $\chi(s)=2$ for $s \geq 4$,
\item $0 \leq \chi'(s) \leq 1$ and $0 \leq -\chi''(s) \leq 1/2$ for all $s$.
\end{itemize}
For any constant $a>0$, we denote $\chi_a:[0,\infty) \to [0,\infty)$ to be the rescaled cutoff function $\chi_a(s):=a \chi(a^{-1} s)$. Note that $0 \leq \chi_a' \leq 1$ and $0 \leq -\chi''_a \leq a^{-1}/2$ everywhere.

\section{Assumptions and main results}
\label{S:main}

\subsection{Assumptions}
\label{SS:assumption}

From now on, we will assume the following for the rest of the paper unless stated otherwise. Let $(\widetilde{M},g)$ be a smooth complete Riemannian manifold with
\begin{itemize}
\item injectivity radius lower bound $\inj_g(\widetilde{M})\geq i_0>0$ and
\item sectional curvature bound $|\sec_{(\widetilde{M},g)}| \leq k_0$ for some $k_0 \geq 0$;
\end{itemize}
and $\Omega \subset \widetilde{M}$ be an open domain with smooth non-empty boundary $\partial \Omega$ such that
\begin{itemize}
\item $\Omega$ has bounded ball curvatures $\max\{|\overline{Z}_\Omega|, |\underline{Z}_\Omega|\} \leq \kappa_0$ for some $\kappa_0 \geq 0$.
\end{itemize}
The outward unit normal of $\partial \Omega$ is denoted by $\nu$. We also fix a bounded open set $U \subset \widetilde{M}$ with Lipschitz boundary. Suppose we have 
\begin{itemize}
\item a sequence $\{\ep_i\}_{i=1}^\infty$ of positive real numbers $\ep_i \to 0$;
\item a sequence $\{\lambda_i\}_{i=1}^\infty$ of real numbers such that $|\lambda_i| \leq \Lambda_0$ for all $i$;
\item a sequence $\{u_i\}_{i=1}^\infty$ of smooth functions on $\overline{\Omega} \cap U$ with $\|u_i\|_{L^\infty} \leq C_0$ and $E_{\ep_i}(u_i) \leq E_0$ satisfying the (volume-constrained) Allen-Cahn equation
 \begin{equation}
\label{A:1a}
-\ep_i \Delta u_i + \frac{W'(u_i)}{\ep_i}=\lambda_i \quad  \text{in $\Omega \cap U$}
\end{equation}
together with the homogeneous Neumann boundary condition
\begin{equation}
\label{A:1b}
\frac{\partial u_i}{\partial \nu}=0 \quad \text{on $\partial \Omega \cap U$}.
\end{equation}
\end{itemize}

\begin{remark}
\label{R:1}
We can relax the a-priori regularity of $u_i$ to lie only inside certain Sobolev spaces. The interior regularity for weak solutions to (\ref{A:1a}) follows from standard elliptic PDE theory \cite{Gilbarg2001} while the boundary regularity along $\partial \Omega$ for weak solutions to \eqref{A:1a} and \eqref{A:1b} follows from \cite{Agmon1959}.
\end{remark}

As in \cite{Hutchinson2000}, after passing to a subsequence, $u_i$ converges a.e. and in $L^1$ to some $u \in BV(\Omega \cap U)$ such that $u = \pm 1$ a.e.. Hence, $\{u=1\}$ and $\{u=-1\}$ are sets of finite perimeter in $U \cap \Omega$ and we write $M:=\partial^* \{u=1\}$ for the reduced boundary of $\{u=1\}$ in $\Omega$ with outward-pointing unit normal $\nu_M$ (see \cite[Section 5.7]{Evans2015}). Moreover, we have
\begin{equation}
\label{E:u-BV}
\|\partial \{u=1\}\|(\Omega \cap U) = \frac{1}{2} \int_{\Omega \cap U} |Du| \leq \frac{E_0}{h_0}.
\end{equation}

For each $i$, we define a Radon measure $\mu_i$ on $U$ by
\begin{equation}
\label{E:mu0}
\mu_i:= \frac{1}{h_0} \left( \frac{\ep_i |\nabla u_i|^2}{2} +\frac{W(u_i)}{\ep_i} \right) \; dv_g \mres {\Omega \cap U},
\end{equation}
and the associated varifold $V_i \in \bV_{n-1}(U)$ by
\begin{equation}
\label{E:asso-varifold}
V_i(\phi):=\int_{\{|\nabla u_i| \neq 0\}} \phi \left(x,I-\frac{\nabla u_i}{|\nabla u_i|} \otimes \frac{\nabla u_i}{|\nabla u_i|} \right)\; d\mu_i
\end{equation}
for $\phi \in C_c(G(U))$. Note that as measures on $U$, both $\mu_i$ and $V_i$ are supported in $U \cap \overline{\Omega}$ and $\|V_i\|=\mu_i \mres {\{| \nabla u_i| \neq 0\}}$. Moreover, the uniform energy bound $E_{\ep_i}(u_i) \leq E_0$ implies the uniform bound $\|V_i\|(U) \leq E_0/h_0$ and hence (after passing to a subsequence) $V_i$ converges as varifolds to some $V \in \bV_{n-1}(U)$. The first variation of $V_i$ is given by (see \cite{Padilla1998})
\begin{equation}
\label{E:Vi-first-variation}
\delta V_i(X)= \int_{\{|\nabla u_i| \neq 0\}} \left\langle \nabla X, I-\frac{\nabla u_i}{|\nabla u_i|} \otimes \frac{\nabla u_i}{|\nabla u_i|} \right\rangle_g \; d\mu_i 
\end{equation}
for every $X \in \mathfrak{X}_c(U)$.

\subsection{Boundary behavior}
\label{SS:boundary}

As in \cite{Kagaya2018}, we can also consider the behavior of $u_i|_{\partial \Omega \cap U}$. First of all, we prove a useful lemma which gives a local bound for the Allen-Cahn energy restricted to the boundary. This can be viewed as a localized version of \cite[Lemma 4.2]{Kagaya2018}.

\begin{lemma}
\label{L:local-boundary-energy}
For any $\tilde{U} \Subset U$, there exists a constant $C_2>0$ depending only on $n$, $k_0$, $\kappa_0$, $\dist_g(\tilde{U},\partial U)$ such that
\[ \begin{aligned}
\int_{\partial \Omega \cap \tilde{U}} \frac{\ep_i |\nabla u_i|^2}{2} + \frac{W(u_i)}{\ep_i} \; d\mathcal{H}^{n-1} \leq &  C_2 \int_{\Omega \cap U} \frac{\ep_i |\nabla u_i|^2}{2} + \frac{W(u_i)}{\ep_i} \\
& +\Lambda_0 C_0(C_2 dv_g(\Omega \cap U) + \mathcal{H}^{n-1}(\partial \Omega \cap U)).
\end{aligned}
\]
\end{lemma}

\begin{proof}
Let $X \in \mathfrak{X}_c(U)$. By the calculations in \cite[Lemma 4.1]{Kagaya2018}, we have
\begin{equation}
\label{E:local-boundary-energy-1}
\begin{aligned}
h_0 \delta V_i(X) = & \int_{\Omega \cap \{|\nabla u_i| \neq 0\}} \nabla X \cdot \left( \frac{\nabla u_i}{|\nabla u_i|} \otimes \frac{\nabla u_i}{|\nabla u_i|} \right) \left( \frac{\ep_i |\nabla u_i|^2}{2} - \frac{W(u_i)}{\ep_i}  \right)  \\
& - \int_{\Omega \cap \{|\nabla u_i| =0\}} (\nabla X \cdot I) \frac{W(u_i)}{\ep_i}  + \lambda_i \int_{\Omega} u_i \Div_g X \\
& + \int_{\partial \Omega} \left( \frac{\ep_i |\nabla u_i|^2}{2} + \frac{W(u_i)}{\ep_i} -\lambda_i u_i \right) (X \cdot \nu) d \mathcal{H}^{n-1}.  
\end{aligned}
\end{equation}
We choose a vector field $X=-\zeta \nabla (\chi_a \circ d)$ where $d(\cdot):=\dist_g(\cdot, \partial \Omega)$, $a=(8 \kappa_0)^{-1}$ and $\zeta \in C_c(U)$ is a smooth cutoff function such that $\zeta=1$ on $\tilde{U}$; $0 \leq \zeta \leq 1$ and $|\nabla \zeta| \leq \frac{2}{\dist_g(\tilde{U},\partial U)}$. Note that $X=\zeta \nu$ along $\partial \Omega \cap U$ and $\|\nabla X\|_{L^\infty} \leq C(n,k_0,\kappa_0,\dist_g(\tilde{U},\partial U))$ by \eqref{E:dist-estimates}. Moreover, by \eqref{E:Vi-first-variation}, it is easily seen that
\[ |h_0 \delta V_i(X)| \leq \|\nabla X\|_{L^\infty}  \int_{\Omega \cap U} \frac{\ep_i |\nabla u_i|^2}{2} + \frac{W(u_i)}{\ep_i}. \]
Combining this and by examining each of the terms on the right hand side of \eqref{E:local-boundary-energy-1}, we have the desired inequality using the that $\|u_i\|_{L^\infty} \leq C_0$ and $|\lambda_i| \leq \Lambda_0$.
\end{proof}

By the same arguments as in \cite{Hutchinson2000} (c.f. \cite[Theorem 3.1]{Kagaya2018}), using Lemma \ref{L:local-boundary-energy}, after passing to a subsequence, $u_i \mres {\partial \Omega \cap U}$ converges a.e. and in $L^1_{loc}$ with respect to $\mathcal{H}^{n-1}$ to some $\tilde{u} \in BV_{loc}(\partial \Omega \cap U)$ such that $\tilde{u} = \pm 1$ $\mathcal{H}^{n-1}$-a.e.. Hence, $\{\tilde{u}=1\}$ and $\{\tilde{u}=-1\}$ are sets of locally finite perimeter in $\partial \Omega \cap U$. However, $\tilde{u}$ may not coincide with the trace of $u$ on $\partial \Omega \cap U$ (see \cite{Kagaya2018} and the discussions therein).

\subsection{Main result}
\label{SS:main}

We are now ready to state the main result of this paper.

\begin{theorem}
\label{T:main}
Under the assumptions and notations given in Section \ref{SS:assumption}, after possibly passing to a subsequence, we have
\begin{itemize}
\item $\lambda_i \to \lambda$, 
\item $u_i \to u$ a.e. and in $L^1(\Omega \cap U)$ with respect to $dv_g$,
\item $u_i \mres {\partial \Omega \cap U} \to \tilde{u}$ a.e. and in $L^1_{loc}(\partial \Omega \cap U)$ with respect to $\mathcal{H}^{n-1}$,
\item $V_i \to V$ as varifolds in $U$,
\end{itemize}
where $\lambda \in \mathbb{R}$, $u \in BV(\Omega \cap U)$, $\tilde{u} \in BV_{loc}(\partial \Omega \cap U)$ and $V \in \bV_{n-1}(U)$.
Moreover, the following statements are true:
\begin{itemize}
\item[(1)] (Equipartition of energy) For each $\phi \in C_c(U)$,
\begin{equation}
\begin{aligned}
\frac{1}{2}\|V\|(\phi)&=\lim_{i \to \infty} \frac{1}{h_0} \int_{\overline{\Omega} \cap U} \frac{\ep_i |\nabla u_i|^2}{2} \phi = \lim_{i \to \infty} \frac{1}{h_0} \int_{\overline{\Omega} \cap U} \frac{W(u_i)}{\ep_i} \phi.
\end{aligned}
\end{equation}
\item[(2)] (Uniform convergence away from limit interface) $M=\partial^* \{u=1\} \subset \spt \|V\| \cap \Omega$ and $u_i \to \pm 1$ locally uniformly on $(\Omega \cap U) \setminus \spt \|V\|$ as $i \to \infty$. For each $b \in (0,1)$, the subsets $\{|u_i| \leq 1-b\}$ locally converges in the Hausdorff sense to $\spt \|V\|$ inside $\Omega \cap U$ as $i \to \infty$.
\item[(3)] (Regularity of limit interface) $V \in \bIV_{n-1}(U)$ with density 
\[ \Theta(x)= \left\{ \begin{array}{cl}
\text{odd} & \text{for $\mathcal{H}^{n-1}$-a.e. $x \in M$},\\
\text{even} & \text{for $\mathcal{H}^{n-1}$-a.e. $x \in (\spt \|V\| \cap \Omega) \setminus M$}.
\end{array} \right.
\]
In other words, $V$ is integer rectifiable up to the boundary $\partial \Omega$ even when $\|V\|(\partial \Omega)>0$.
\item[(4)] (Variational property of limit interface) $V$ has free boundary in $\partial \Omega \cap U$ with
\[ H^T(x)= \left\{ \begin{array}{cl}
\frac{2 \lambda}{h_0 \Theta(x)} \nu_M(x)  & \text{for $\mathcal{H}^{n-1}$-a.e. $x \in M$},\\
0 & \text{for $\mathcal{H}^{n-1}$-a.e. $x \in \spt \|V\|  \setminus M$}.
\end{array} \right.
\]
Moreover, $\spt \sigma_V \subset \partial \Omega \cap U$ and $\eta(x) =\nu(x)$ for $\sigma_V$-a.e. $x$ on the generalized boundary of $V$.
\end{itemize}
\end{theorem}

\begin{remark}
When $U \subset \Omega$, Theorem \ref{T:main} reduces to the interior regularity result in \cite{Hutchinson2000} (see also \cite{Guaraco2018,Mantoulidis2021} for the adaption to the Riemannian setting).
\end{remark}

\begin{remark}
One key conclusion of Theorem \ref{T:main} is that the limit interface $V$ is integer-rectifiable \emph{up to the boundary} $\partial \Omega$. To the best of the authors' knowledge, the boundary behavior is only partially treated under further restrictive assumptions. For example, the first three conclusions in Theorem \ref{T:main} were obtained in \cite{Tonegawa2003} in the case of convex Euclidean domains. In their recent work \cite{Kagaya2018}, Kagaya-Tonegawa considered the more general Allen-Cahn functional with boundary contact energy and studied the boundary behavior of the limit interface \emph{under the assumption} of vanishing discrepancy measure. In this paper, we prove, in the case of zero boundary contact energy, the vanishing of the discrepancy measure without any convexity assumption on $\partial \Omega$, even when boundary concentration of $\|V\|$ occurs. We will address the more general case with boundary contact energy in a future work.
\end{remark}

\begin{remark}
As in \cite{Tonegawa2002,Tonegawa2005}, one can consider variable chemical potentials lying in suitable Sobolev spaces instead of constants on the right hand side of \eqref{A:1a}. We will treat this in a forthcoming work and study the validity of the Gibbs-Thomson law as in \cite{Roeger2008}.
\end{remark}

\begin{remark}
By \cite[Theorem 1.3]{Masi2020}, we can conclude that the $(n-2)$-dimensional part of the generalized measure $\sigma_V \lc \{x|\Theta^{n-2}_*(\sigma_V,x)>0\}$ is $(n-2)$-rectifiable. It would be interesting to study more refined regularity and singular structure of the full generalized measure $\sigma_V$, which is highly non-trivial due to the lack of a suitable monotonicity formula (c.f. \cite[Section 4]{Tonegawa2003}). We will address this in future work.

\end{remark}

\section{Monotonicity formula and consequences}
\label{S:monotonicity}

Let $(\widetilde{M},g)$ and $\Omega, U \subset \widetilde{M}$ be as given in Section \ref{SS:assumption}. Throughout this section, we assume $u \in C^\infty(\overline{\Omega} \cap U)$ is a solution to 
 \begin{equation}
\label{E:1a}
-\ep \Delta u + \frac{W'(u)}{\ep}=\lambda \quad  \text{in $\Omega \cap U$,}
\end{equation}
\begin{equation}
\label{E:1b}
\frac{\partial u}{\partial \nu}=0 \quad \text{on $\partial \Omega \cap U$},
\end{equation}
where $\epsilon \in (0,1)$, $|\lambda| \leq \Lambda_0$, $\|u\|_{L^\infty} \leq C_0$ and $E_\ep(u) \leq E_0$. We define the energy density $e_\epsilon(u)$ and discrepancy function $\xi_\ep(u)$ respectively by
\begin{equation}
\label{E:energy-density}
e_\ep(u):= \frac{\ep |\nabla u|^2}{2} + \frac{W(u)}{\ep}
\end{equation}
and 
\begin{equation}
\label{E:discrepancy}
\xi_\ep(u):= \frac{\ep |\nabla u|^2}{2} - \frac{W(u)}{\ep}.
\end{equation}

Our first goal is to prove an almost-monotonicity inequality (c.f. \cite[Appendix B]{Guaraco2018}) of the form:
\[ \frac{d}{d\rho} \left( e^{c\rho} \rho^{-(n-1)} \int_{B_\rho(x) \cap \Omega} e_\ep(u) \right) \geq  e^{c\rho} \rho^{-(n-1)} \int_{B_\rho(x) \cap \Omega}  -(\xi_\ep(u))^+ \]
for some $c>0$ and all sufficiently small $\rho$. Next, we establish up-to-the-boundary estimates, in both pointwise and integral forms, for $u$ and $\xi_\ep(u)$. Together with the almost-monotonicity inequality, these estimates allow us to control the energy density ratio
\begin{equation}
\label{E:energy-ratio}
I_{\ep,u}(r,x):=\frac{1}{\omega_{n-1} r^{n-1}} \int_{B_r(x) \cap \Omega} e_\ep(u) 
\end{equation}
which will be instrumental in deriving the regularity of the limit interface in later sections.

\subsection{Almost monotonicity formula}
\label{SS:monotonicity}

Recall that the monotonicity formula for minimal surfaces is obtained by applying the first variation formula for the area functional with (suitably truncated) radial test vector fields. In the Allen-Cahn setting, the almost monotonicity formula (which can be viewed as an $\ep$-version of the monotonicity formula for minimal surfaces) can be obtained by similar arguments. The following \emph{Pohozaev identity} plays the role of the first variation formula in this regard.

\begin{lemma}[Pohozaev identity]
\label{L:Pohozaev} 
Let $u \in C^2(\overline{\Omega} \cap U)$ be a solution to the semilinear equation $\Delta u = F(u)$ in $\Omega \cap U$ where $F=G'$ for some smooth function $G:\mathbb{R} \to \mathbb{R}$. For any $X \in \mathfrak{X}_c(U)$, we have
\begin{equation}
\begin{aligned}
\label{E:Pohozaev} 
\int_{\Omega \cap U} \left[ \left(\frac{|\nabla u|^2}{2} + G(u)  \right) \Div_g X - \nabla X (\nabla u, \nabla u) \right] \; dv_g  & \\
=\int_{\partial \Omega \cap U}  \left[ \left(\frac{|\nabla u|^2}{2} + G(u)  \right) (X \cdot \nu)  -  (\nabla u \cdot X) \frac{\partial u}{\partial \nu} \right] & \; d\mathcal{H}^{n-1}.
\end{aligned}
\end{equation}
\end{lemma}

\begin{proof}
Multiply both sides of the equation $\Delta u = F(u)$ by $X \cdot \nabla u$ and then integrate by part twice (see \cite[Lemma 9]{Xiong2018} for a similar calculation). 
\end{proof}

For the Allen-Cahn equation \eqref{E:1a}, we choose $G(t)=\ep^{-2} \widetilde{W}(t)$ where
\[ \widetilde{W}(t):=W(t)- \ep \lambda t + \ep \Lambda_0 C_0.\]
Note that the constant term $\ep \Lambda_0 C_0$ is there simply to ensure that $\widetilde{W}(u) \geq 0$ under our assumptions. For this choice of $G$ and using the Neumann boundary condition \eqref{E:1b}, the Pohozaev identity \eqref{E:Pohozaev} reads
\begin{equation}
\begin{aligned}
\label{E:Pohozaev-1} 
\int_{\Omega \cap U} \left[ \left(\frac{\ep |\nabla u|^2}{2} + \frac{\widetilde{W}(u)}{\ep} \right) \Div_g X - \ep \nabla X (\nabla u, \nabla u) \right] \; dv_g &  \\
=\int_{\partial \Omega \cap U} \left(\frac{\ep |\nabla u|^2}{2} + \frac{\widetilde{W}(u)}{\ep}\right) (X \cdot \nu) \; d\mathcal{H}^{n-1} &.
\end{aligned}
\end{equation}

We define $\widetilde{e}_\ep(u)$, $\widetilde{\xi}_\ep(u)$ and $\widetilde{I}_{\ep,u}(r,x)$ as in \eqref{E:energy-density}, \eqref{E:discrepancy} and \eqref{E:energy-ratio} similarly with $W$ replaced by $\widetilde{W}$. We will first establish the almost monotonicity formula for $\widetilde{I}_{\ep,u}$ from which the corresponding almost monotonicity formula for $I_{\ep,u}$ follows.

For the purpose of making the test vector fields compactly supported, we let $\varphi: \mathbb{R} \to \mathbb{R}$ be any smooth decreasing cutoff function such that $\varphi(t)=1$ for $t \leq \frac{1}{2}$ and $\varphi(t) =0$ for $t \geq 1$. For any $x \in \widetilde{M}$ and $0<\rho < \inj_g(x)$, the function $\varphi \left( \frac{\dist_g(x,\cdot)}{\rho} \right)$ is a smooth function which is compactly supported inside $B_\rho(x)$.

From now on, we fix $x \in U \cap \overline{\Omega}$ and denote $r(\cdot):=\dist_g(x,\cdot)$. For any $0 < \rho < \min\{\dist_g(x,\partial U),i_0\}$, we consider the test vector field $X \in \mathfrak{X}_c(U)$ defined by
\[ X:= \varphi \left( \frac{r}{\rho} \right) r \nabla r \]
which is a truncated radial vector field centered at $x$ supported inside $B_\rho(x)$. When $\rho < \frac{1}{\sqrt{k_0}}$, a direct computation together with \eqref{E:Hessian-comparison} and \eqref{E:Laplacian-comparison} gives 
\[ \left| \Div_g X(x) -  \left(n \varphi  +\frac{r}{\rho} \varphi' \right) \right| \leq (n-1) \sqrt{k_0} r \varphi \] 
and
\[ \left| \nabla X(\nabla u, \nabla u) - \left( |\nabla u|^2 \varphi + \frac{r}{\rho} (\nabla u \cdot \nabla r)^2 \varphi' \right) \right| \leq \sqrt{k_0} r |\nabla u|^2 \varphi \]
where $\varphi$ and $\varphi'$ are both evaluated at $\frac{r}{\rho}$ by abuse of notation. Observing the relation
\[ \rho \frac{d}{d\rho} \left[  \varphi\left( \frac{r}{\rho} \right) \right] = - \frac{r}{\rho} \varphi' \left( \frac{r}{\rho} \right).\]
and plugging the above estimates into the left hand side of the Pohozaev identity \eqref{E:Pohozaev-1}, we obtain the inequality
\begin{equation}
\label{E:Pohozaev-2}
\begin{aligned}
\frac{d}{d\rho} & \left[ \rho^{-(n-1)} \int_{\Omega \cap U} \widetilde{e}_\ep(u) \varphi \right] -\rho^{-(n-1)} \frac{d}{d\rho} \left[ \int_{\Omega \cap U} \ep (\nabla u \cdot \nabla r)^2 \varphi \right] + \rho^{-n} \int_{\Omega \cap U} \widetilde{\xi}_\ep(u) \varphi \\
\geq  & - (n+1) \sqrt{k_0} \left[  \rho^{-(n-1)} \int_{\Omega \cap U} \widetilde{e}_\ep(u) \varphi \right] - \rho^{-(n-1)} \int_{\partial \Omega \cap U} \widetilde{e}_\ep(u) |\nabla r \cdot \nu| \varphi \; d\mathcal{H}^{n-1}.
\end{aligned}
\end{equation} 
Note that we have used the fact that $r\leq \rho$ in the support of $\varphi$ and $\ep |\nabla u|^2 \leq 2 \widetilde{e}_\ep(u)$. We would like to consider two separate cases:
\begin{itemize}
\item \underline{Case 1:} $x \in \Omega \cap U$ and $B_\rho(x) \Subset \Omega \cap U$;
\item \underline{Case 2:} $x \in \partial \Omega \cap U$ $B_\rho(x) \Subset U$.
\end{itemize}

In the first case, there is no boundary term as $B_\rho(x) \cap \partial \Omega = \emptyset$. We obtain the following interior almost-monotonicity inequality (c.f. \cite[Appendix B]{Guaraco2018}), which reduces to \cite[Lemma 3.1]{Hutchinson2000} when $k_0=0$.
\begin{lemma}[Interior almost-monotonicity inequality]
\label{L:interior-mono-1}
For any $x \in \Omega \cap U$ and $0<\rho<  \min\{i_0,\frac{1}{\sqrt{k_0}}\}$ such that $B_\rho(x) \Subset \Omega \cap U$, we have
\begin{equation}
\label{E:interior-mono-1}
\frac{d}{d\rho} \left[ e^{(n+1)\sqrt{k_0} \rho} \rho^{-(n-1)} \int_{B_\rho(x)} \widetilde{e}_\ep(u) \right]  \geq e^{(n+1)\sqrt{k_0} \rho} \rho^{-n} \int_{B_\rho(x)} -\widetilde{\xi}_\ep(u) .
\end{equation}
\end{lemma}

\begin{proof}
It follows directly from \eqref{E:Pohozaev-2} by dropping the non-negative term and the boundary term. Finally, we let $\varphi \to \mathbf{1}_{(-\infty,1)}$.
\end{proof}

For the second case, we need to handle the extra boundary term in \eqref{E:Pohozaev-2}. The following lemma allows us to rewrite the boundary term in terms of an interior integral (and its derivative). 

\begin{lemma}
\label{L:boundary-estimate}
For any $x \in \partial \Omega \cap U$ and $0<\rho<  \min\{i_0,\frac{1}{2\kappa_0}\}$ such that $B_\rho(x) \Subset U$, we have
\begin{equation}
\label{E:boundary-estimate}
\int_{\partial \Omega \cap B_\rho(x)} \widetilde{e}_\ep(u) \; d \mathcal{H}^{n-1} \leq (n+1)C_2 \int_{\Omega \cap B_\rho(x)} \widetilde{e}_\ep(u) + 6 \frac{d}{d\rho} \left[ \int_{\Omega \cap B_\rho(x)} \widetilde{e}_\ep(u)  \right].
\end{equation}
\end{lemma}

\begin{proof}
Let $\varphi$ be the radial cutoff function centered at $x$ as defined before. For any $0 < \rho < \min\{\dist_g(x,\partial U),\frac{1}{2\kappa_0}\}$, we can consider the test vector field $X \in \mathfrak{X}_c(U)$ defined by 
\[ X=-\varphi \left( \frac{r}{\rho} \right) \nabla d\]
where $d$ is the signed distance function from $\partial \Omega$ as defined in section \ref{SS:extrinsic-geom}. Note that $X$ is smooth and compactly supported in $B_\rho(x)$ with $X \cdot \nu=\varphi$ along $\partial \Omega$. By \eqref{E:dist-estimates}, we have the following estimates:
\[ \Div_g X \leq  (n-1) C_1 \varphi   - \frac{1}{\rho} \varphi' \]
\and
\[ -\ep \nabla X(\nabla u, \nabla u) \leq \left( C_1 \varphi - \frac{1}{\rho} \varphi' \right) \ep | \nabla u|^2.\]
Putting all these into the Pohozaev identity \eqref{E:Pohozaev-1}, noting that $r \geq \frac{\rho}{2}$ in the support of $\varphi'$ and $\ep |\nabla u|^2 \leq 2 \widetilde{e}_\ep(u)$, and then letting $\varphi \to \mathbf{1}_{(-\infty,1)}$, we obtain the desired inequality.
\end{proof}

We are now ready to establish the almost-monotonicity inequality centered at a boundary point $x \in \partial \Omega$.

\begin{lemma}[Boundary almost-monotonicity inequality]
\label{L:boundary-mono-1}
There exists a constant 
\[c_1=c_1(n,k_0,\kappa_0)\geq 0\]
so that the following holds: for any $x \in \partial \Omega \cap U$ and $0<\rho<  \frac{1}{2} \min\{i_0,\frac{1}{\kappa_0}\}$ such that $B_\rho(x) \Subset U$, we have
\begin{equation}
\label{E:boundary-mono-1}
\frac{d}{d\rho} \left[ e^{c_1 \rho} \rho^{-(n-1)} \int_{\Omega \cap B_\rho(x)} \widetilde{e}_\ep(u) \right]  \geq \frac{e^{c_1 \rho}}{1+3 \kappa_0 \rho} \rho^{-n} \int_{\Omega \cap B_\rho(x)} -\widetilde{\xi}_\ep(u) .
\end{equation}
\end{lemma}

\begin{proof}
First of all, note that by \eqref{E:allard-estimate}, we have the following estimate on the boundary term of \eqref{E:Pohozaev-2} for any $\rho<\frac{1}{2} i_0$, 
\[ \rho^{-(n-1)} \int_{\partial \Omega \cap B_\rho(x)} \widetilde{e}_\ep(u) |\nabla r \cdot \nu| \; d\mathcal{H}^{n-1} \leq \frac{\kappa_0}{2} \rho^{-(n-2)} \int_{\partial \Omega \cap B_\rho(x)} \widetilde{e}_\ep(u) \; d\mathcal{H}^{n-1}.  \]
On the other hand, Lemma \ref{L:boundary-estimate} implies that when $\rho < \min\{i_0,\frac{1}{2\kappa_0}\}$, we have the inequality
\begin{equation*}
\begin{aligned}
6 \rho \frac{d}{d \rho} \left[ \rho^{-(n-1)} \int_{\Omega \cap B_\rho(x)} \widetilde{e}_\ep(u) \right] \geq &  \rho^{-(n-2)} \int_{\partial \Omega \cap B_\rho(x)} \widetilde{e}_\ep(u) \; d\mathcal{H}^{n-1} \\
& - [(n+1)C_1 \rho + 6(n-1)] \rho^{-(n-1)} \int_{\Omega \cap B_\rho(x)} \widetilde{e}_\ep (u) . 
\end{aligned}
\end{equation*}
Putting it back to \eqref{E:Pohozaev-2}, using $\rho \kappa_0 < \frac{1}{2}$, after letting $\varphi \to \mathbf{1}_{(-\infty,1)}$ and rearranging terms gives the desired inequality with $c_1=\frac{n+1}{4} C_1 + 3(n-1) \kappa_0$. 
\end{proof}

\begin{remark}
    As noted in the Introduction, it would be possible to mimic the approach of \cite{Grueter1986} and \cite{Tonegawa2003}, by reflecting across the boundary $\partial\Omega$, and prove an almost-monotonicity inequality for the sum of the energy density ratio on a ball centered in $\overline{\Omega}$ and its reflection. We chose, for the sake of simplicity, to avoid reflection across the boundary entirely, and consider balls that are either disjoint from $\partial\Omega$, or centered at points lying on $\partial\Omega$, as in \cite{Guang2020}.
\end{remark}

\subsection{Auxiliary estimates}
\label{SS:auxiliary}

We first derive some a-priori estimates for any solution $u$ to \eqref{E:1a} and \eqref{E:1b}. These estimates hold locally uniformly in space \emph{up to the boundary} $\partial \Omega$ and uniformly in $\ep$. We fix throughout this subsection an open subset $\tilde{U} \Subset U$.

We begin with the $C^0$-estimate. Our proof uses a maximum principle argument similar to that in \cite[Proposition 3.2]{Hutchinson2000}, except that we have to include an extra barrier term to prevent the extrema from occurring along the boundary.

\begin{lemma}[Uniform $C^0$ estimate]
\label{L:u-C0-bound}
There exist constants 
\[ \ep_2=\ep_2(\Lambda_0, C_0, W, n, i_0, k_0, \kappa_0, \dist_g(\tilde{U},\partial U))>0,\]
\[ c_2=c_2(\Lambda_0, C_0, W, n, i_0, k_0, \kappa_0, \dist_g(\tilde{U},\partial U))\geq 1\]
such that for all $0 < \ep < \ep_2$, we have 
\begin{equation}
\label{E:u-C0-bound}
 \|u\|_{C^0(\overline{\Omega} \cap \tilde{U})} \leq 1 +c_2 \ep  <2
\end{equation}
\end{lemma}

\begin{proof}
Let $x \in \overline{\Omega} \cap \tilde{U}$ and $\rho< \frac{1}{6} \min\{i_0, \frac{1}{\sqrt{k_0}}, \frac{1}{\kappa_0}, \dist_g(x,\partial U)\}$. Denote $r(\cdot):=\dist_g(x,\cdot)$. Inside the ball $B_{3 \rho}(x) \Subset U$, we fix a smooth radial function $\zeta=\zeta(r)$ such that 
\begin{itemize}
\item $\zeta = 1+\frac{c_2 \ep}{2}$ inside $B_\rho(x)$;
\item $\zeta = 1+C_0$ on the annulus $B_{3\rho} (x) \setminus B_{2\rho} (x)$; and
\item inside the ball $B_{3 \rho}(x)$, we have by \eqref{E:Laplacian-comparison}
\[ 1+\frac{c_2 \ep}{2} \leq \zeta \leq 1+C_0, \quad |\nabla \zeta| \leq 2C_0 \rho^{-1}, \quad |\Delta \zeta| \leq C(n) C_0 \rho^{-2}\]
\end{itemize}
where $c_2$ will be chosen later and then $\ep$ is sufficiently small so that $c_2 \ep < \min\{1,C_0\}$. 

We first claim that $u \leq 1+c_2 \ep$ on $B_\rho(x) \cap \Omega$ for all sufficiently small $\rho$ by a contradiction argument. Suppose we have $\sup_{B_\rho(x) \cap \Omega} u > 1+c_2 \ep$. We divide into two cases as before:
\begin{itemize}
\item \underline{Case 1:} $x \in \Omega \cap \tilde{U}$ and $B_{3\rho}(x) \Subset \Omega \cap U$;
\item \underline{Case 2:} $x \in \partial \Omega \cap \tilde{U}$ and $B_{3\rho}(x) \Subset U$.
\end{itemize}
For the first case, we consider as in \cite[Proposition 3.2]{Hutchinson2000} the auxiliary function $f:=u-\zeta$ which is smooth inside $B_{3 \rho}(x)$. One can verify
\begin{itemize}
\item $f \leq -1$ along the boundary $\partial B_{3 \rho}(x)$, and
\item $\sup_{B_{3\rho}(x)} f > \frac{c_2 \ep}{2}>0$.
\end{itemize}
Therefore, $f(x_0)=\sup_{B_{3\rho}(x)} f> \frac{c_2 \ep}{2}$ at an interior maxima $x_0 \in B_{3\rho}(x)$. By maximum principle, at $x_0$ we have
\begin{equation}
\begin{aligned}
0 \geq  \ep \Delta f = & \ep^{-1} W'(u) -\lambda -\ep \Delta \zeta \\
= & \ep^{-1} \left. W'(tu+(1-t)\zeta) \right|_{t=0}^{t=1} + \ep^{-1} W'(\zeta) - \lambda -\ep \Delta \zeta   \\
\geq & \ep^{-1} f \int_0^1  W''(tu+(1-t) \zeta)  \; dt - \lambda -\ep \Delta \zeta \\
\geq &\frac{c_2 \kappa}{2} - \lambda - \ep \Delta \zeta ,
\end{aligned}
\end{equation}
Note that we have used that $u > \zeta \geq 1$ at $x_0$ and the properties of $W$ in section \ref{SS:double-well}. If $c_2 \geq 1$ is sufficiently large, depending only on $\Lambda_0$, $C_0$, $W$, $n$, $i_0$, $k_0$, $\dist_g(\tilde{U},\partial U)$ and $\dist_g(x,\partial \Omega)$, we arrive at a contradiction.

For the second case, we consider another auxiliary function
\[ f:=u -\zeta + \frac{4 C_0}{\rho} \chi_a \circ d \]
where $d(\cdot)=\dist_g(\cdot, \partial \Omega)$ and $\chi$ is the cutoff function defined in section \ref{SS:cutoff} with $a=\frac{\ep \rho}{16 C_0}$. One can again verify
\begin{itemize}
\item $f \leq -1+\frac{\ep}{2}< -\frac{1}{2}$ along the boundary $\partial B_{3 \rho}(x) \cap \Omega$;
\item $\sup_{B_{3\rho}(x) \cap \Omega} f > \frac{c_2 \ep}{2}>0$; and 
\item $\frac{\partial f}{\partial \nu}\leq -2 C_0 \rho^{-1}<0$ along $B_{3\rho}(x) \cap \partial \Omega$.
\end{itemize}
Hence, $\sup_{B_{3\rho}(x) \cap \Omega} f$ is again achieved by an interior maxima $x_0 \in B_{3\rho}(x) \cap \Omega$. Applying the maximum principle at $x_0$, a similar calculation as before (note that we still have $u\geq 1$ and $\zeta>1$ by our choice of $a$) gives
\[ \frac{c_2 \kappa}{2} \leq \lambda+ \ep \Delta \zeta + \frac{1}{2} \max_{t \in [1,1+C_0]} W''(t) -\frac{4 \ep C_0}{\rho} \Delta( \chi_a \circ d).\]
By the definition of $\chi_a$ in section \ref{SS:cutoff} and \eqref{E:dist-estimates}, we will again obtain a contradiction when $c_2$ is sufficiently large, depending only on $\Lambda_0$, $C_0$, $W$, $n$, $i_0$, $k_0$, $\kappa_0$ and $\dist_g(\tilde{U},\partial U)$. This proves the desired upper bound for $u$. The lower bound can be obtained similarly by considering the auxiliary function $f=u+\zeta$ for Case 1 and $f=u+\zeta-\frac{4 C_0}{\rho} \chi_a \circ d$ for Case 2. 
\end{proof}

Using the $C^0$-estimate in Lemma \ref{L:u-C0-bound}, one obtains up-to-the-boundary gradient estimates by standard PDE theory of semi-linear elliptic equations.

\begin{lemma}[Uniform gradient estimates]
\label{L:apriori-u-bound}
There exists a constant
\[ c_3=c_3(\Lambda_0, C_0, W, n, i_0, k_0, \kappa_0, \dist_g(\tilde{U},\partial U))\geq 1\]
such that for all $0 < \ep < \ep_2$, we have the gradient estimate 
\begin{equation}
\label{E:apriori-du-bound}
\| \ep \nabla u\|_{C^0(\overline{\Omega} \cap \tilde{U})} \leq c_3.
\end{equation}
\end{lemma}

\begin{proof}
The interior gradient estimate can be found e.g. in \cite[Proposition 2.19]{Han2011}. The boundary gradient estimates along $\partial \Omega$ follows e.g. from \cite[Theorem 1.1]{Xu2017} (c.f. \cite{Xu2017a} for the extension to the Riemannian setting). 
\end{proof}


We can actually obtain better estimates away from the transition region $\{|u| \leq \alpha \}$ which imply that $\frac{W(u)}{\ep}$ is of order $\mathcal{O}(\ep^{2\eta-1})$ for any $\eta \in (0,1)$.

\begin{lemma}[Refined uniform $C^0$ estimates away from the transition region]
\label{L:refine-estimate}
For any $\eta \in (0,1)$, there exist constants 
\[ \ep_4=\ep_4(\eta,\Lambda_0, C_0, W, n, i_0, k_0, \kappa_0, \dist_g(\tilde{U},\partial U))>0,\]
\[ c_4=c_4(\Lambda_0, C_0, W, n, i_0, k_0, \kappa_0, \dist_g(\tilde{U},\partial U))\geq 1\]
such that for all $0 < \ep < \ep_4$, $x \in \overline{\Omega} \cap \tilde{U}$ and $\ep^{1-\eta} \leq r \leq \frac{1}{6} \min\{i_0, \frac{1}{\sqrt{k_0}}, \frac{1}{\kappa_0}, \dist_g(x,\partial U)\}$, we have
\begin{itemize}
\item[(i)] if $u \geq \alpha$ on $B_{3r}(x) \cap \Omega$, then $1-c_4 \ep^\eta \leq u \leq 1 +c_4 \ep$ on $B_r(x) \cap \Omega$;
\item[(ii)] if $u \leq -\alpha$ on $B_{3r}(x) \cap \Omega$, then $-1-c_4 \ep \leq u \leq -1 +c_4 \ep^\eta$ on $B_r(x) \cap \Omega$.
\end{itemize}

\end{lemma}

\begin{proof}
We only prove (i) as the arguments are the same for (ii). Let $x \in \overline{\Omega} \cap \tilde{U}$ and $r<\frac{1}{6} \min\{i_0, \frac{1}{\sqrt{k_0}}, \frac{1}{\kappa_0}, \dist_g(x,\partial U)\}$. Suppose $u \geq \alpha$ on $B_{3r}(x) \cap \Omega$. By Lemma \ref{L:u-C0-bound}, we have $u \leq 1+c_2 \ep$ provided that $\ep < \ep_2$. It remains to prove that $u \geq 1- c_4 \ep^\eta$ on $B_r(x) \cap \Omega$ for some suitable choice of $c_4$, range of $r$ and sufficiently small $\ep$. 

We will argue by contradiction as in the proof of Lemma \ref{L:u-C0-bound}. Suppose, on the contrary, that $\alpha \leq \inf_{B_r(x) \cap \Omega} u < 1-c_4 \ep^\eta$. We fix a smooth radial function $\zeta$ such that 
\begin{itemize}
\item $\zeta = 1-\frac{c_4 \kappa \ep^\eta}{4C}$ inside $B_r(x)$;
\item $\zeta = \alpha +\frac{c_4 \kappa \ep^\eta}{4C}$ on the annulus $B_{3r} (x) \setminus B_{2r} (x)$; and
\item inside the ball $B_{3 r}(x)$, we have by \eqref{E:Laplacian-comparison}
\[ \alpha +\frac{c_4 \kappa \ep^\eta}{4C} \leq \zeta \leq 1-\frac{c_4 \kappa \ep^\eta}{4C}, \quad |\nabla \zeta| \leq 2r^{-1}, \quad |\Delta \zeta| \leq C(n) r^{-2}\]
\end{itemize}
where $C=C(W):=\max\{1,\max_{s \in [\alpha,2]} |W''(s)|\}$, $c_4$ will be chosen later and then $\ep$ is sufficiently small so that $c_4 \ep^\eta < \min\{1,\frac{2C(1-\alpha)}{\kappa}\}$. 

We now consider the auxiliary function
\[ f:=u -\zeta + \frac{4}{r} \chi_a \circ d \]
where $d(\cdot)=\dist_g(\cdot, \partial \Omega)$ and $\chi$ is the cutoff function defined in section \ref{SS:cutoff} with $a=\frac{c_4 \kappa r \ep^\eta }{16 C}$. One can verify (note that $C \geq \kappa$)
\begin{itemize}
\item $f \geq -\frac{3}{4} \frac{c_4 \kappa \ep^\eta}{C} \geq -\frac{3}{4} c_4 \ep^\eta$ along the boundary $\partial B_{3r}(x) \cap \Omega$;
\item $\inf_{B_{3r}(x) \cap \Omega} f < -\frac{3}{4} c_4 \ep^\eta$; and 
\item $\frac{\partial f}{\partial \nu}\geq \frac{2}{r}>0$ along $B_{3r}(x) \cap \partial \Omega$, if non-empty.
\end{itemize}
Hence, no matter whether $x \in \Omega$ or $x \in \partial \Omega$, $\inf_{B_{3r}(x) \cap \Omega} f$ is achieved by an interior minima $x_0 \in B_{3r}(x) \cap \Omega$. Applying the maximum principle at $x_0$, by our choice of $a$ and a similar calculation as before (note that we still have $\alpha \leq u(x_0) < 2$ and $\alpha \leq \zeta(x_0) \leq 1$) gives
\[ \frac{3 c_4 \kappa}{4} \ep^{-1+\eta} \leq -\lambda- \ep \Delta \zeta  -\frac{4 \ep}{r} \Delta( \chi_a \circ d).\]
By the definition of $\chi_a$ in section \ref{SS:cutoff} and \eqref{E:dist-estimates}, this implies
\[ \frac{3 \kappa}{4} c_4 \leq \Lambda_0 \ep^{1-\eta} +C(n) \ep^{2-\eta} r^{-2} + 4(n-1)C_2 \ep^{2-\eta} r^{-1} + \frac{32C}{c_4 \kappa} \ep^{2-2\eta} r^{-2}.  \] 
For $r \geq \ep^{1-\eta}$, this yields a contradiction when $c_4$ is sufficiently large, depending only on $\Lambda_0$, $W$, $n$, $i_0$, $k_0$, $\kappa_0$ and $\dist_g(\tilde{U},\partial U)$.
\end{proof}

\begin{remark}
In case $B_{3r}(x) \Subset \Omega$, we do not need to include the term involving $\chi_a$ in the auxiliary function $f$ and we have the stronger conclusion that $u \geq 1-c_4 \ep$ (i.e. with $\eta=1$), c.f. \cite[Proposition 4.1]{Hutchinson2000}. With a different choice of auxiliary function $\zeta$, we can also establish the same estimates with $\eta=1$.
However, unlike in \cite{Hutchinson2000} (see the remark after Proposition 3.2), apparently one cannot obtain better estimates even when $\lambda=0$ by iterating the arguments above.
\end{remark}

\subsection{Estimates for the discrepancy function}
\label{SS:discrepancy}

Note that in order for the almost-monotonicity formula in Lemma \ref{L:interior-mono-1} and \ref{L:boundary-mono-1} to be effective, we need some control on the right hand side of the inequalities involving the discrepancy function $\xi_\ep$ defined in \eqref{E:discrepancy}. 

We give in this subsection a crude pointwise upper bound for the discrepancy function. We want to point out that our bound is weaker than the interior bound obtained in \cite[Proposition 3.3]{Hutchinson2000} but holds up to the boundary $\partial \Omega$. Note that we have chosen to work at the original scale, rather than rescaling by $\ep$ as in \cite{Hutchinson2000}. This does not affect our subsequent analysis, but certain expressions will change by appropriate multiples of $\ep$.

\begin{proposition}[Pointwise upper bound on discrepancy function]
\label{P:xi-bound}
There exists a constant
\[ \ep_5=\ep_5(\Lambda_0, C_0, W, n, i_0, k_0, \kappa_0, \dist_g(\tilde{U},\partial U))>0,\]
such that for all $0 < \ep < \ep_5$, we have 
\begin{equation}
\label{E:xi-bound}
\sup_{\overline{\Omega} \cap \tilde{U}} \xi_{\ep} (u) \leq  \ep^{-4/5}.
\end{equation}
\end{proposition}

\begin{proof}
We consider as in \cite{Hutchinson2000} the following auxiliary function
\begin{equation}
\label{E:xi_G}
\xi_{G}:=\frac{\ep}{2} |\nabla u|^2 -\ep^{-1} W(u)- \ep^{-1}G(u)
\end{equation}
where $G\in C^\infty(\mathbb{R})$ is to be chosen later. By Bochner's formula \eqref{E:Bochner} and a similar computation as in \cite[Lemma 3.5]{Hutchinson2000}, we have the following differential inequality at points where $|\nabla u|>0$:
\begin{equation}
\label{E:xi-inequality}
\begin{aligned}
\ep^2  & \Delta \xi_{G} - \frac{2(W'+G') \nabla u}{|\nabla u|^2}  \cdot \nabla \xi_{G}   +2(G''+(n-1)k_0 \ep^2) \xi_{G} \\
 \geq &  \ep^{-1} (G')^2 + \ep^{-1} G'W' -2 \ep^{-1} (G''+(n-1)k_0 \ep^2) (W+G) +\lambda (W'+G')
\end{aligned}
\end{equation}
where the functions $G$ and $W$ and their derivatives are evaluated at $u$. 

Let $\delta \in (0,1)$ to be fixed later and define (recall Lemma \ref{L:u-C0-bound})
\[ G=G_\delta(r):=\delta \left( 1+ \int_{-1-c_2 \ep_2}^r \exp \left(-\int_{-1-c_2 \ep_2}^t \frac{|W'(s)|+\delta}{2(W(s)+\delta)}\; ds \right) \; dt \right). \]
Noting that $\frac{W}{W'}$ is a smooth function which vanishes at $\pm 1$, one easily verifies that the following holds on $r \in [-1-c_2 \ep_2,1+c_2 \ep_2]$,
\begin{itemize}
\item $\delta \leq G \leq 5 \delta$,
\item $C^{-1} \delta^2< G' \leq \delta$,
\item $0 < -G'' =G' \frac{|W'(r)|+\delta}{2(W(r)+\delta)} \leq C$
\end{itemize}
for some constant $C=C(W)>1$. By Lemma \ref{L:u-C0-bound}, we have (the functions $G$ and $W$ and their derivatives are evaluated at $u$)
\begin{equation}
\label{E:G'-estimate}
G' W'-2G''(W+G) = G'\left( W'+\frac{|W'|+\delta}{W+\delta} (W+G) \right) \geq \delta G'
\end{equation}
provided that $\ep <\ep_2$. 

Let $x \in \overline{\Omega} \cap \tilde{U}$ and $R=\frac{1}{6} \min\{i_0, \frac{1}{\sqrt{k_0}}, \frac{1}{\kappa_0}, \dist_g(x,\partial U),1\}$. We fix a smooth radial cutoff function $\zeta$ centered at $x$ such that
\begin{itemize}
\item $\zeta = 1$ on $B_{R/8}(x)$
\item $\zeta =0$ outside $B_{R/4}(x)$,
\item $0 \leq \zeta \leq 1$, $|\nabla \zeta| \leq 8 R^{-1} \zeta^{1/2}$, $|\Delta \zeta| \leq C(n) R^{-2}$
\end{itemize}
where $C(n)>0$ is a constant depending only on $n$. Let $d(\cdot):=\dist_g(\cdot, \partial \Omega)$ and $\chi$ be the cutoff function defined in section \ref{SS:cutoff} with $a=(8\hat{C})^{-1} \ep^{1/5} R$ where $\hat{C}>0$ is a constant to be fixed later. Define a new auxiliary function
\[ f:=\zeta \xi_{G} + \ep^{-1} R^{-1} \hat{C} \chi_a \circ d .\]
We claim that for sufficiently small $\ep$, we have
\begin{equation}
\label{E:rescaled-xi-bound-cutoff}
\sup_{\Omega} f \leq \frac{1}{2} \ep^{-4/5}.
\end{equation}
If the supremum is non-positive, then we must have $\xi_G \leq 0$ and by our choice of $G$ that
\begin{equation}
\xi_\ep = \xi_{G}+\ep^{-1} G(u) \leq 5 \ep^{-1} \delta
\end{equation}
and thus $\xi_\ep(x) \leq \ep^{-4/5}$ provided that $\delta \leq \frac{1}{5} \ep^{1/5}$. So, without loss of generality, we assume that $\sup_{\Omega} f >0$ and is achieved at some point $x_0 \in B_{R/4}(x) \cap \overline{\Omega}$ (note that $f \leq \frac{1}{4} \ep^{-4/5}$ outside $B_{R/4}(x)$ by our choice of $a$). First, we claim that $x_0 \notin \partial \Omega$. By the Neumann boundary condition \eqref{E:1b} and \cite[Lemma 3.1]{Tonegawa2003}, along $\partial \Omega$ we have
\[ \frac{\partial  f}{\partial \nu} = \frac{\partial \zeta}{\partial \nu} \xi_G - \ep \zeta B(\nabla u, \nabla  u) - \ep^{-1} R^{-1} \hat{C}  \]
where $B$ is the second fundamental form of $\partial \Omega$ with respect to $\nu$ (i.e. $B(\tau,\tau):=\langle \nabla_\tau \nu, \tau \rangle$). Since $\kappa_0$ bounds the maximum absolute value of the principal curvatures of $\partial \Omega$, by Lemma \ref{L:apriori-u-bound} and our choice of $\zeta$, we have $\frac{\partial  f}{\partial \nu} <0$ along $\partial \Omega$ by taking $\hat{C}:=10 c_3^2$. Hence, we can assume without loss of generality that $x_0 \notin \partial  \Omega$. Suppose \eqref{E:rescaled-xi-bound-cutoff} does not hold, then at an interior maximum $x_0$ we have by our choice of $a$ that
\begin{equation}
\label{E:rescaled-xi-bound-cutoff-fail}
\frac{1}{2}\ep \zeta |\nabla u|^2 \geq \zeta  \xi_{G} >  \frac{1}{2} \ep^{-4/5} -  2a \hat{C} \epsilon^{-1} R^{-1}  = \frac{1}{4} \ep^{-4/5}.
\end{equation}
From this and \eqref{E:apriori-du-bound}, we have at $x_0$,
\begin{equation}
\label{E:xi-zeta-bound}
\xi_G \leq \frac{c_3^2}{2} \ep^{-1} \quad \text{ and } \quad \zeta^{-1} \leq 2 c_3^2 \ep^{-1/5}.
\end{equation}
Since $x_0$ is an interior maximum point of $f$, we have $\nabla f =0$ and $\Delta f \leq 0$ at $x_0$. The first order condition says $\zeta \nabla \xi_G = - \xi_G \nabla \zeta - \ep^{-1} R^{-1} \hat{C} \chi_a'  \nabla d$, which together with \eqref{E:xi-zeta-bound} implies 
\[ \zeta^{1/2} |\nabla \xi_G| \leq 4c_3^2 \ep^{-1} R^{-1} +  \sqrt{2} \hat{C} c_3 \ep^{-11/10} R^{-1} \leq C(c_3) \ep^{-11/10} R^{-1}.\]
On the other hand, the second order condition at $x_0$ says
\[ \zeta \Delta \xi_G \leq -\xi_G \Delta \zeta -2 \nabla \zeta \cdot \nabla \xi_G - \ep^{-1} R^{-1} \hat{C} \chi_a' \Delta d - \ep^{-1} R^{-1} \hat{C} \chi_a'' |\nabla d|^2.\]
Using the first order condition, our choice of $\zeta$, together with \eqref{E:dist-estimates} and \eqref{E:xi-zeta-bound}, we have
\[ \begin{aligned} 
\Delta \xi_G & \leq  \Big( - \xi_G \zeta^{-1} \Delta \zeta +2 \xi_G \zeta^{-2} |\nabla \zeta|^2  \Big) + 2 \zeta^{-2} \ep^{-1} R^{-1}\hat{C} \chi_a' (\nabla d \cdot \nabla \zeta) \\
& \hspace{2cm}  - \ep^{-1} R^{-1} \hat{C} \zeta^{-1} \chi_a' \Delta d - \ep^{-1} R^{-1} \hat{C} \zeta^{-1} \chi_a'' |\nabla d|^2 \\
& \leq C(n,c_3) \ep^{-6/5} R^{-2} + C(c_3) \ep^{-13/10} R^{-2} \\
& \hspace{2cm} + C(C_1,c_3,n) \ep^{-6/5} R^{-1} + C(c_3) \ep^{-7/5} R^{-2} \\
& \leq C(C_1,c_3,n) \ep^{-7/5} R^{-2}. \\
\end{aligned} \]

Combining the estimates above, for sufficiently small $\ep$ depending only on $C_1$, $n$, $\hat{C}$ and $c_3$, we have at $x_0$
\begin{equation}
\textrm{(L.H.S.) of \eqref{E:xi-inequality}} \leq  C(C_1,c_3,n) \ep^{3/5}R^{-2} + C(W,c_3) \ep^{-1/5} R^{-1}+ C(n,k_0,c_3) \ep  
\end{equation}
On the other hand, by our choice of $G$, together with \eqref{E:G'-estimate}
\begin{equation}
\begin{aligned}
\textrm{(R.H.S.) of \eqref{E:xi-inequality}} & \geq  \ep^{-1} (G')^2 + \ep^{-1} \delta G' -C(n,k_0,W) \ep - C(W) \Lambda_0  \\
&\geq \ep^{-1}(C^{-2} \delta^4 + C^{-1} \delta^3) -C(n,k_0,W) \ep - C(W) \Lambda_0 \\
&\geq 2 C^{-2} \ep^{-1} \delta^4 -C(n,k_0,W) \ep - C(W) \Lambda_0.
\end{aligned}
\end{equation}
Combining these we arrive at the inequality
\[ \delta^4 \leq C'' \ep^{4/5}\]
for some constant $C''>0$ depending only on $\Lambda_0$, $C_0$, $W$, $n$, $i_0$, $k_0$, $\kappa_0$ and $\dist_g(\tilde{U},\partial U)$. Therefore, we arrive at a contradiction by taking $\delta= ( 2C'')^{1/4} \ep^{1/5} <1$ for sufficiently small $\ep$. Therefore, \eqref{E:rescaled-xi-bound-cutoff} holds and Proposition \ref{P:xi-bound} follows.
\end{proof}

\begin{remark}
    With a different choice of the auxiliary function $G$, resembling the one in \cite{Kagaya2019}, it is possible to obtain a better estimate of order $\ep^{-3/5}$ in Proposition \ref{P:xi-bound}.
\end{remark}

\subsection{Consequences of the almost monotonicity formula}
\label{SS:consequences}

We now combine the results obtained in the previous subsections to derive some consequences of the almost monotonicity formula (Lemma \ref{L:interior-mono-1} and \ref{L:boundary-mono-1}). We first establish a uniform lower bound of the energy density ratio \eqref{E:energy-ratio} on balls of radius $\ep \leq r \leq \mathcal{O} (\ep^{4/5})$ centered at a boundary point within the transition region $\{|u| \leq 1-b\}$.

\begin{lemma}[Energy density ratio lower bound I]
\label{L:density-lower-bound-1}
For any $b \in (0,1-\alpha)$, there exist constants
\[ \ep_6=\ep_6(b,\Lambda_0, C_0, W, n, i_0, k_0, \kappa_0, \dist_g(\tilde{U},\partial U))>0,\]
\[ c_6=c_6(b, \Lambda_0, C_0, W, n, i_0, k_0, \kappa_0, \dist_g(\tilde{U},\partial U)) > 0\]
such that for all $0<\ep <\ep_4$, $x \in \partial \Omega \cap \tilde{U}$ with $|u(x)| \leq 1-b$, and $\ep \leq r \leq c_6 \ep^{4/5}$, we have
\[ I_{\ep,u}(r,x) \geq  c_6. \]
In fact, the same holds for $x \in \Omega \cap \tilde{U}$ provided furthermore that $r<\dist_g(x,\partial \Omega)$. 
\end{lemma}

\begin{proof}
We first assume $x \in \partial \Omega$. Since $|u(x)| \leq 1-b$, we have $|u| \leq 1-\frac{b}{2}$ on $B_{\frac{b \ep}{2c_3}}(x) \cap \Omega$ by \eqref{E:apriori-du-bound} provided that $\ep<\ep_2$. As a result, by \eqref{E:half-ball-comparison} and \eqref{E:energy-ratio} we have for sufficiently small $\ep$ that
\begin{equation}
\label{E:density-lower-1}
\widetilde{I}_{\ep,u}(\ep,x) \geq I_{\ep,u}(\ep,x) \geq \frac{1}{\omega_{n-1}\ep^{n-1}} \int_{B_{\frac{b \ep}{2c_3}}(p) \cap \Omega} \frac{W(u)}{\ep}  \geq 4  c'_6
\end{equation}
where $c'_6:= \frac{1}{12} \frac{\omega_n}{\omega_{n-1}} \left( \frac{b}{2c_3} \right)^n  \min_{|s| \leq 1-\frac{b}{2}} W(s)>0$.

Let $c_6'':= \frac{\omega_{n-1}}{\omega_n} c_6'$. Assuming $\ep < \ep_5$ is sufficiently small such that $c_6'' \ep^{4/5} < \frac{1}{2} \min\{\dist(\tilde{U},\partial U),i_0, \frac{1}{\kappa_0}\}$, we can apply the boundary almost monotonicity formula \eqref{E:boundary-mono-1} and the upper bound on discrepancy function \eqref{E:xi-bound} to see that for any $\ep \leq r \leq c_6'' \ep^{4/5}$ (note that $\widetilde{\xi}_\ep (u) \leq \xi_\ep(u)$ by definition)
\begin{equation}
\begin{aligned}
\widetilde{I}_{\ep,u}(r,x) & \geq e^{-c_1 (r-\ep)} \widetilde{I}_{\ep,u}(\ep,x) - \frac{2}{3} \frac{\omega_n}{\omega_{n-1}}  \ep^{-4/5} (r-\ep) \\
& \geq e^{-c_1c_6' \ep^{4/5}} \widetilde{I}_{\ep,u}(\ep,x) - \frac{2}{3} \frac{\omega_n}{\omega_{n-1}}  c_6'' > 2 c_6'
\end{aligned}
\end{equation}
for $\ep$ sufficiently small such that \eqref{E:density-lower-1} holds and $e^{-c_1c_6' \ep^{4/5}} \geq \frac{3}{4}$. Setting $c_6:=\min\{c_6',c_6''\}$, the asserted lower bound for $I_{\ep,u}(r,x)$ follows from the estimate
\begin{equation}
\label{E:I-tilde-estimate}
\widetilde{I}_{\ep,u}(r,x) \leq I_{\ep,u}(r,x) + \frac{4}{3} \frac{\omega_n}{\omega_{n-1}} \Lambda_0 C_0 r 
\end{equation}
which holds for all sufficiently small $r$ depending only on $k_0$ and $\kappa_0$. The case for interior points $x \in \Omega$ can be obtained similarly using the interior monotonicity formula \eqref{E:interior-mono-1}.
\end{proof}

In order to prove a uniform lower bound of the energy density ratio on balls of radius $r \geq \mathcal{O}(\ep^{4/5})$, the pointwise upper bound in Proposition \ref{P:xi-bound} will not be sufficient for our purpose. We need a more refined integral upper bound for the discrepancy function. It is important to note that the center $x$ in the proposition below may or may not lie in $\partial \Omega$, and does not need to be inside the transition region $\{|u| \leq 1-b\}$.

\begin{proposition}[Integral upper bound for discrepancy function]
\label{P:xi-integral-bound}
For any $b \in (0,1-\alpha)$, there exist constants
\[ r_2=r_2(n, i_0, k_0, \kappa_0, \dist_g(\tilde{U},\partial U)) \in (0,1),\]
\[ \ep_7=\ep_7(b,\Lambda_0, C_0, W, n, i_0, k_0, \kappa_0, \dist_g(\tilde{U},\partial U))>0,\]
\[ c_7=c_7(b,\Lambda_0, C_0, W, n, i_0, k_0, \kappa_0, \dist_g(\tilde{U},\partial U)) > 1\]
such that for any $0<\ep <\ep_7$, $x \in \overline{\Omega} \cap \tilde{U}$ and $c_6 \ep^{4/5} \leq r \leq  r_2$, we have
\begin{equation}
\label{E:discrepancy-integral-bound}
\frac{1}{r^n} \int_{B_r(x) \cap \Omega} \xi_{\ep}^+ (u) \leq \frac{c_7}{r^{7/8}} (I_{\ep,u}(r,x)+1).
\end{equation}
\end{proposition}

\begin{proof}
Let $b \in (0,1-\alpha)$. Fix $x \in \overline{\Omega} \cap \tilde{U}$ and $c_6 \ep^{4/5} \leq r \leq  r_2$ where $r_2$ is sufficiently small with $r_2 < \frac{1}{2} \{r_0,r_1,\dist_g(x,\partial U)\}$. We divide the domain of integration $B_r(x) \cap \Omega$ into a union of three disjoint subsets as follow:
\[ \mathcal{A}:=\Big((B_r(x) \cap \Omega) \setminus B_{r-2\ep^{9/10}}(x) \Big) \cup \Big((B_r(x) \cap \Omega)\cap U_{2\ep^{9/10}}(\partial \Omega)\Big), \]
\[ \mathcal{B}:=\left\{ y \in (B_r(x) \cap \Omega) \setminus \mathcal{A} \; : \; B_{\ep^{9/10}}(y) \cap \left\{ |u| \leq 1-b \right\}  \neq \emptyset \right\}, \]
\[ \mathcal{C}:=\left\{ y \in (B_r(x) \cap \Omega) \setminus \mathcal{A} \; : \; B_{\ep^{9/10}}(y) \cap \left\{ |u| \leq 1-b \right\}  = \emptyset \right\} .\]
Note that $r \geq c_6 \ep^{4/5} \geq  2\ep^{9/10}$ for all sufficiently small $\ep$ (depending on $c_6$). We will estimate the (L.H.S.) of \eqref{E:discrepancy-integral-bound} over each of the three subsets defined above.

We begin with the estimate on $\mathcal{A}$. By \eqref{E:sphere-comparison} and the bound on ball curvatures, for $r_2$ small enough (depending only on $n$, $k_0$ and $\kappa_0$) we have the following volume estimate
\[ dv_g(\mathcal{A}) \leq C(n) r^{n-1} \ep^{9/10}.\]
Combining this estimate with \eqref{E:xi-bound} and $r \geq c_6 \ep^{4/5}$, we have
\[ \frac{1}{r^n} \int_{\mathcal{A}} \xi_\ep^+(u)  \leq C(n)  \ep^{1/10} r^{-1} \leq C(n) c_6^{-1/8} r^{-7/8} \]
which is clearly bounded by the (R.H.S.) of \eqref{E:discrepancy-integral-bound} provided that $c_7 \geq C(n) c_6^{-1/8}$.

We now proceed with the estimate on $\mathcal{B}$. We first need to an estimate of the volume $dv_g(\mathcal{B})$ by considering the covering of $\mathcal{B}$ by balls $\{B_{\ep^{9/10}}(z)\}$ with centers $z \in \left\{ |u| \leq 1-b \right\} \cap (\Omega \cap B_{r-\ep^{9/10}}(x)\setminus \overline{U_{\ep^{9/10}}(\partial \Omega)})$. By Vitali covering lemma, we can extract a finite disjoint family of such balls $\{B_{\ep^{9/10}}(z_i)\}_{i=1}^N$ so that $\mathcal{B} \subset \cup_{i=1}^N B_{5\ep^{9/10}}(z_i)$. Hence, we have from \eqref{E:ball-comparison} that the volume estimate
\begin{equation}
\label{E:B-vol-estimate}
dv_g (\mathcal{B}) \leq N \cdot C(n) \ep^{9n/10}
\end{equation}
holds for all sufficiently small $\ep$ depending on $r_0$. It remains to bound $N$ from above. Applying Lemma \ref{L:density-lower-bound-1} (with $x=z_i$), since $\ep \leq \ep^{9/10} \leq c_6 \ep^{4/5}$ and $\dist(z_i,\partial \Omega) > \ep^{9/10}$, we have $I_{\ep,u}(\ep^{9/10},z_i) \geq c_6$. Summing over all $i$, noting that $\{B_{\ep^{9/10}}(z_i)\}_{i=1}^N$ are disjoint balls contained in $B_r(x) \cap \Omega$, we have 
\[ N \omega_{n-1} c_6 \ep^{9(n-1)/10} \leq \int_{B_r(x) \cap \Omega} e_\ep(u)=\omega_{n-1} r^{n-1} I_{\ep,u}(r,x).\]
This gives an upper estimate on $N$. Putting it back to \eqref{E:B-vol-estimate}, we obtain the volume estimate
\[ dv_g(\mathcal{B}) \leq  C(n,c_6) \ep^{9/10} r^{n-1} I_{\ep,u}(r,x).  \]
Using this estimate together with \eqref{E:xi-bound} and $r \geq c_6 \ep^{4/5}$, we have
\[ \frac{1}{r^n} \int_{\mathcal{B}} \xi_\ep^+(u)  \leq  C(n,c_6) \ep^{1/10} r^{-1} I_{\ep,u}(r,x) \leq C(n,c_6) r^{-7/8} I_{\ep,u}(r,x) \]
which is clearly bounded by the (R.H.S.) of \eqref{E:discrepancy-integral-bound} with $c_7$ sufficiently large depending only on $n$ and $c_6$.

Finally, it remains to prove the estimate on $\mathcal{C}$. Consider the Lipschitz function
\[ \phi(y):=\min \left\{1, \ep^{-9/10} \dist_g \left(y, (\tilde{U}\setminus (B_r(x) \cap \Omega))  \cup \left\{|u| \leq 1-b \right\} \right) \right\} \]
which clearly satisfies
\begin{itemize}
\item $\phi=0$ on $(\tilde{U} \setminus(B_r(x) \cap \Omega)) \cup \left\{|u| \leq  1-b \right\}$,
\item $\phi=1$ on $\mathcal{C}$,
\item $0 \leq \phi \leq 1$ and $|\nabla \phi| \leq \ep^{-9/10}$ everywhere.
\end{itemize}
Taking gradient on both sides of \eqref{E:1a} and then multiplying by $\phi^2 \nabla u$, we have by Bochner formula \eqref{E:Bochner} that
\[ \ep \phi^2 \left( \frac{1}{2} \Delta |\nabla u|^2 - |\Hess u|^2 +(n-1)k_0 |\nabla u|^2\right)  \geq \ep \phi^2 \langle \nabla \Delta u, \nabla u \rangle = \frac{W''(u)}{\ep} \phi^2 |\nabla u|^2. \]
Integrating the above on $B_r(x) \cap \Omega$ and integrating by part (note that $\phi=0$ on $B_r(x) \cap \partial \Omega$), we obtain
\[ \begin{aligned}
& \int_{B_r(x) \cap \Omega} \left( \ep \phi^2 |\Hess u|^2 +\frac{W''(u)}{\ep} \phi^2 |\nabla u|^2 \right) \\
 &\leq   - \ep \int_{B_r(x) \cap \Omega}  \phi \langle \nabla \phi, \nabla |\nabla u|^2 \rangle + C(n,k_0) \ep \int_{B_r(x) \cap \Omega} \phi^2 |\nabla u|^2  \\
&=  - 2\ep \int_{B_r(x) \cap \Omega}   \phi \Hess u (\nabla u, \nabla \phi) + C(n,k_0) \ep \int_{B_r(x) \cap \Omega} \phi^2 |\nabla u|^2  \\
&\leq  \frac{1}{2} \int_{B_r(x) \cap \Omega}  \ep \phi^2 |\Hess u|^2 +2 \int_{B_r(x) \cap \Omega}  \ep |\nabla u|^2 |\nabla \phi|^2  + C(n,k_0) \ep \int_{B_r(x) \cap \Omega} \phi^2 |\nabla u|^2
\end{aligned}
\]
where we have used Cauchy-Schwarz and Young's inequalities in the last inequality. Since $\spt \phi \subset \left\{|u| > 1-b \right\}$, we have $W''(u) \geq \kappa$ on $\spt \phi$. Absorbing the first term on right to the left hand side and combining this with $|\nabla \phi| \leq \ep^{-9/10}$, we obtain for $\ep$ sufficiently small (depending on $n$, $k_0$ and $\kappa$)
\[ \frac{\kappa}{2\ep}  \int_{B_r(x) \cap \Omega}  \phi^2 |\nabla u|^2 \leq 2 \ep^{-4/5}  \int_{B_r(x) \cap \Omega} |\nabla u|^2 .   \]
Since $\phi=1$ on $\mathcal{C}$ and $0 \leq \phi \leq 1$, together with $r \geq c_6 \ep^{4/5}$, we have
\[ \frac{1}{r^n} \int_{\mathcal{C}} \xi_\ep^+(u)  \leq \frac{4\omega_{n-1}}{\kappa} \ep^{1/5} r^{-1} I_{\ep,u}(r,x) \leq  C(n,W,c_6)  r^{-3/4} I_{\ep,u}(r,x)  \]
which is clearly bounded by the (R.H.S.) of \eqref{E:discrepancy-integral-bound} with $c_7$ sufficiently large depending only on $n$, $W$ and $c_6$ (note that $r^{-3/4} \leq r^{-7/8}$). 
\end{proof}

Using Proposition \ref{P:xi-integral-bound}, we can extend Lemma \ref{L:density-lower-bound-1} beyond the scale of $\mathcal{O}(\ep^{4/5})$.

\begin{proposition}[Energy density ratio lower bound II]
\label{P:density-lower-bound-3}
For any $b \in (0,1-\alpha)$, there exist constants
\[ r_3=r_3(b,\Lambda_0, C_0, W, n, i_0, k_0, \kappa_0, \dist_g(\tilde{U},\partial U)) \in (0,1),\]
\[ \ep_8=\ep_8(b,\Lambda_0, C_0, W, n, i_0, k_0, \kappa_0, \dist_g(\tilde{U},\partial U))>0,\]
\[ c_8=c_8(b,\Lambda_0, C_0, W, n, i_0, k_0, \kappa_0, \dist_g(\tilde{U},\partial U)) > 0\]
such that for all $0<\ep <\ep_8$, $x \in \partial \Omega \cap \tilde{U}$ with $|u(x)| \leq 1-b$, and $\ep \leq r \leq r_3$, we have
\[ I_{\ep,u}(r,x) \geq  c_8. \]
In fact, the same holds for $x \in \Omega \cap \tilde{U}$ provided furthermore that $r<\dist_g(x,\partial \Omega)$. 

\end{proposition}

\begin{proof}
By Lemma \ref{L:density-lower-bound-1}, the proposition holds for $\ep \leq r \leq c_6 \ep^{4/5}$ with $c_8=c_6$ and $\ep_8=\ep_6$. The point is to use the almost monotonicity formula in Lemma \ref{L:interior-mono-1} and \ref{L:boundary-mono-1} together with Proposition \ref{P:xi-integral-bound} to show that $I_{\ep,u}(r,x)$ cannot further decrease too much as $r$ increases beyond the scale $c_6 \ep^{4/5}$ to some definite radius $r_3$ to be determined. 

We first deal with the case $x \in \partial \Omega$. By the boundary almost monotonicity formula Lemma \ref{L:boundary-mono-1} together with \eqref{E:discrepancy-integral-bound}, for $\ep <\ep_7$ we obtain for $c_6 \ep^{4/5} \leq \rho \leq  \frac{1}{2} \min\{r_2, i_0,\frac{1}{\kappa_0},\dist_g(x,\partial \tilde{U})\}$, 
\begin{equation}
\label{E:mono-boundary}
\frac{d}{d\rho}\left(  \omega_{n-1} e^{c_1 \rho} \widetilde{I}_{\ep,u}(\rho,x) \right) \geq - \frac{c_7}{\rho^{7/8}} e^{c_1 \rho} ( \widetilde{I}_{\ep,u}(\rho,x)+1).
\end{equation}
We define $r_3$ as the supremum of all $r' \in \left[c_6 \ep^{4/5},\frac{1}{2} \min\{r_2, i_0,\frac{1}{\kappa_0},\dist_g(x,\partial \tilde{U})\} \right]$ such that
\[ e^{c_1 \rho}  \widetilde{I}_{\ep,u}(\rho,x)  \geq \frac{1}{2} c_6 e^{c_1 c_6 \ep^{4/5}} \text{ for all $\rho \in [c_6 \ep^{4/5},r']$}. \]
Note that the left hand side is a continuous function of $\rho$. By Lemma \ref{L:density-lower-bound-1}, we have $r_3>c_6 \ep^{4/5}$. If $r_3=\frac{1}{2} \min\{r_2, i_0,\frac{1}{\kappa_0},\dist_g(x,\partial \tilde{U})\}$, then we are done by taking $\ep_8=\min\{\ep_6,\ep_7\}$ and $c_8=\frac{1}{4}  c_6 e^{-c_1 r_2}$ (we further reduce $r_3$ depending on $k_0$, $\kappa_0$, $C_0$ and $\Lambda_0$ according to the estimate \eqref{E:I-tilde-estimate}). Otherwise, by continuity we must have
\begin{equation}
\label{E:r_0-achieve}
e^{c_1 r_3}  \widetilde{I}_{\ep,u}(r_3,x)  = \frac{1}{2} c_6 e^{c_1 c_6 \ep^{4/5}}.
\end{equation}
By the definition of $r_3$ and \eqref{E:mono-boundary}, we have for $\rho \in [c_6 \ep^{4/5},r_3]$ 
\[ \frac{d}{d\rho} \log \left( e^{c_1 \rho} \widetilde{I}_{\ep,u}(\rho,x)  \right) \geq  -C(n,c_6,c_7,r_2) \rho^{-7/8}. \]
Integrating $\rho$ over $[c_6 \ep^{4/5},r_3]$, using \eqref{E:r_0-achieve} and Lemma \ref{L:density-lower-bound-1},
\[ \log \frac{1}{2} \geq \log \left( \frac{e^{c_1 r_3} \widetilde{I}_{\ep,u}(r_3,x) }{e^{c_1 c_6 \ep^{4/5}} \widetilde{I}_{\ep,u}(c_6 \ep^{4/5},x) } \right) \geq -8C(n,c_6,c_7,r_2) r_3^{1/8}. \]
This gives the desired positive lower bound for $r_3$ independent of $\ep$ and $x$. The proof for the interior case $x \in \Omega$ is similar where we use Lemma \ref{L:interior-mono-1} instead.
\end{proof} 

We now proceed to establish an upper bound for the energy density ratio in terms of the uniform energy bound $E_0$, within the same range of radii as in Proposition \ref{P:density-lower-bound-3}. Note that we do not require the center point $x$ to lie inside the transition region $\{|u| \leq 1-b\}$.

\begin{proposition}[Energy density ratio upper bound]
\label{P:density-upper-bound}
For any $b \in (0,1-\alpha)$, there exist constants
\[ \ep_9=\ep_9(b,\Lambda_0, C_0, W, n, i_0, k_0, \kappa_0, \dist_g(\tilde{U},\partial U))>0,\]
\[ c_9=c_9(b,\Lambda_0, C_0, W, n, i_0, k_0, \kappa_0, \dist_g(\tilde{U},\partial U)) > 1\]
such that for all $0<\ep <\ep_9$, $x \in \partial \Omega \cap \tilde{U}$, and $\ep \leq r \leq r_3$, we have
\[ I_{\ep,u}(r,x) \leq  c_9(E_0+1). \]
In fact, the same holds for $x \in \Omega \cap \tilde{U}$ provided furthermore that $r<\dist_g(x,\partial \Omega)$. 
\end{proposition}

\begin{proof}
We only give the proof for $x \in \partial \Omega$ as the interior case is similar. We first treat the case when $c_6 \ep^{4/5} \leq r \leq r_3$. Let $\delta_0:=\max\{1,e^{c_1 r_3} \widetilde{I}_{\ep,u}(r_3,x)\}>0$. Note that
\begin{equation}
\label{E:delta0}
\delta_0 \leq C(c_1,n) r_3^{-(n-1)} E_0 + C(c_1,n) \Lambda_0 C_0 \leq C(c_1,n,r_3,\Lambda_0,C_0) (E_0+1).  
\end{equation}
Note that $e^{c_1 r} \widetilde{I}_{\ep,u}(r,x) \leq \delta_0$ clearly implies the desired upper bound of $I_{\ep,u}(r,x)$. We can assume without loss of generality that $e^{c_1 \rho} \widetilde{I}_{\ep,u}(\rho,x) \geq \delta_0$ for all $\rho \in [r,r_3]$. Integrating \eqref{E:mono-boundary} over $[r,r_3]$, we obtain as in Proposition \ref{P:density-lower-bound-3} that
\[  \log \left( \frac{e^{c_1 r_3} \widetilde{I}_{\ep,u}(r_3,x)  }{e^{c_1 r} \widetilde{I}_{\ep,u}(r,x)  } \right) \geq - C(c_7,n,c_1,r_3). \]
This implies 
\[ \widetilde{I}_{\ep,u}(r,x) \leq C(c_7,n,c_1,r_3) \delta_0 \]
which is the desired upper bound by \eqref{E:delta0}.

It remains to treat the case $\ep \leq r \leq c_6 \ep^{4/5}$. Integrating the almost monotonicity formula \eqref{E:boundary-mono-1} over $[r,c_6 \ep^{4/5}]$ and using the pointwise bound \eqref{E:xi-bound}, we obtain
\[ e^{c_1 r} \widetilde{I}_{\ep,u}(r,x) \leq C(n,c_1,r_3,c_6) +e^{c_1 c_6 \ep^{4/5}} \widetilde{I}_{\ep,u}(c_6 \ep^{4/5},x) \]
which yields the desire result as $\widetilde{I}_{\ep,u}(c_6 \ep^{4/5},x) \leq c_9(E_0+1)$.
\end{proof}

We conclude this section with the following proposition which says that the energy density ratio is almost monotone modulo a term of order $r^{1/8}$ for the range of radii in $[c_6 \ep^{4/5},r_3]$. 

\begin{proposition}[Almost monotonicity of energy density ratio]
\label{P:almost-monotonicity}
For any $b \in (0,1-\alpha)$, there exist constants
\[ \ep_{10}=\ep_{10}(b,\Lambda_0, C_0, W, n, i_0, k_0, \kappa_0, \dist_g(\tilde{U},\partial U))>0,\]
\[ c_{10}=c_{10}(b,E_0,\Lambda_0, C_0, W, n, i_0, k_0, \kappa_0, \dist_g(\tilde{U},\partial U)) > 1\]
such that for any $0<\ep <\ep_{10}$, $x \in \partial \Omega \cap \tilde{U}$ and $c_6 \ep^{4/5} \leq s \leq r \leq r_3$, we have
\begin{equation}
\label{E:concentric-mono-1}
I_{\ep,u}(r,x) \geq e^{-c_{10}(r-s)} I_{\ep,u}(s,x) -c_{10} r^{1/8} ;
\end{equation}
In fact, the same holds for $x \in \Omega \cap \tilde{U}$ provided furthermore that $r<\dist(x,\partial \Omega)$. 
\end{proposition}

\begin{proof}
We again just provide the proof for the case $x \in \partial \Omega$. Using Proposition \ref{P:density-upper-bound} and \eqref{E:mono-boundary}, for $\ep < \min\{\ep_7,\ep_9\}$, we have for $\rho \in [c_6 \ep^{4/5},r_3]$ that
\[ \frac{d}{d\rho}\left( e^{c_1 \rho} \widetilde{I}_{\ep,u}(\rho,x) \right) \geq  -C(c_1,c_7,c_9,E_0) \rho^{-7/8}. \]
Integrating $\rho$ over $[s,r]$ gives 
\[ \widetilde{I}_{\ep,u}(r,x) \geq e^{-c_1 (r-s)} \widetilde{I}_{\ep,u}(s,x) -   C(c_1,c_7,c_9,E_0) r^{1/8},\]
from which \eqref{E:concentric-mono-1} easily follows by \eqref{E:I-tilde-estimate} (since $r \leq r^{1/8}$). 
\end{proof}

\section{Density ratio bounds and rectifiability of the limit varifold}
\label{S:rectifiable}

In this section, we assume $u_i$, $\ep_i$ and $\lambda_i$ are sequences given in Section \ref{SS:assumption}. Recall that the associated Radon measures $\mu_i$ on $U$ and the associated varifolds $V_i \in \bV_{n-1}(U)$ are defined in \eqref{E:mu0} and \eqref{E:asso-varifold} respectively. By the compactness of Radon measures, after possibly passing to a subsequence, there exists a Radon measure $\mu$ on $U$ and a varifold $V \in \bV_{n-1}(U)$ such that 
\[ \mu_i \to \mu \qquad \text{ and } \qquad V_i \to V \]
in the sense of measures. Note that $\|V_i\|=\mu_i \mres {\{| \nabla u_i| \neq 0\}}$ from the definition and observe that the support (as measures on $U$) of both $\mu$ and $\|V\|$ is contained in $\overline{\Omega} \cap U$. The main result of this section is to show that $V$ is rectifiable in $U$ and that $\|V\|=\mu$. Note that the rectifiability of $V$ in the interior of $\Omega$ was established in \cite{Hutchinson2000,Guaraco2018}. Our main point here is to prove that the same continues to hold up to the boundary $\partial \Omega$. 

\subsection{Density ratio bound}
\label{SS:density-bound}

We first begin with a lemma which says that any point in $\spt \mu$ can be approximated by a sequence of points $x_i$ in the transition region $\{|u| \leq \alpha \}$.

\begin{lemma}
\label{L:Hausdorff}
For each $x \in \spt \mu$, there exists (after possibly passing to a subsequence) a sequence $x_i \in (\overline{\Omega} \cap U) \cap \{|u_i|\leq \alpha\}$ such that $x_i \to x$.
\end{lemma}

\begin{proof}
We argue by contradiction as in the proof of Proposition 4.1 in \cite{Hutchinson2000}. Let $x \in \spt \mu$. We first assume $x \in \partial \Omega$. Suppose, on the contrary, that there exists some $\overline{r}>0$ such that $B_{\overline{r}}(x) \Subset U$ and $|u_i| >  \alpha $ on $B_{\overline{r}}(x) \cap \overline{\Omega}$ for all large $i$. By shrinking $\overline{r}$ such that $B_{\overline{r}}(x) \cap \overline{\Omega}$ is connected and passing to a subsequence, we can suppose $u_i > \alpha$. By Lemma \ref{L:refine-estimate} (shrinking $\overline{r}$ further if necessary) with a fixed $\eta \in (\frac{1}{2},1)$, we have for all large $i$ that $|u_i-1| \leq c_4 \ep_i^\eta$ and hence $\ep_i^{-1}W(u_i) = \mathcal{O}(\ep_i^{2\eta-1})$ on $B_{\overline{r}/3}(x) \cap \Omega$. We have from Bochner formula \eqref{E:Bochner} that for all large $i$ (such that $(n-1)k_0 \leq \frac{\kappa}{2\ep_i}$)
\[ \Delta \left( \frac{\ep_i |\nabla u_i|^2}{2} \right) \geq  \frac{\kappa}{2\ep_i} |\nabla u_i|^2 . \]
Multiplying the above by any nonnegative $\phi \in C^2_c(B_{\overline{r}}(x) \cap \overline{\Omega})$ and integrating by part, we obtain 
\begin{equation}
\begin{aligned}
& \int_{B_{\overline{r}}(x) \cap \Omega} \phi \frac{\ep_i |\nabla u_i|^2}{2} \\
\leq & \frac{\ep_i^2}{\kappa} \int_{B_{\overline{r}}(x) \cap \Omega} (\Delta \phi)  \frac{\ep_i |\nabla u_i|^2}{2} + \frac{\ep_i^2}{\kappa} \int_{B_{\overline{r}}(x) \cap \partial \Omega}  \left( \ep_i \phi B(\nabla u_i,\nabla u_i) - \frac{\partial \phi}{\partial \nu}   \frac{\ep_i |\nabla u_i|^2}{2} \right) d\mathcal{H}^{n-1}\\
\leq & \frac{E_0}{\kappa} \ep_i^2 \|\phi\|_{C^2}  + \frac{1+2 \kappa_0}{\kappa} \ep_i^2 \|\phi\|_{C^1} \int _{B_{\overline{r}}(x) \cap \partial \Omega} e_{\ep_i}(u_i) \; d\mathcal{H}^{n-1}
\end{aligned}
\end{equation}
where $B$ is the second fundamental form of $\partial \Omega$. By the proof of Lemma \ref{L:local-boundary-energy} together with the estimates $\ep_i^{-1}W(u_i) = \mathcal{O}(\ep^{2\eta-1})$ with $\eta> 1/2$ and letting $i \to \infty$, we have $\mu (B_{\overline{r}/3}(x)) =0$, contradicting that $x \in \spt \mu$. The case $x \notin \partial \Omega$ is indeed simpler. We further shrink $\overline{r}$ such that $B_{2\overline{r}}(x) \Subset \Omega \cap U$ and there are no boundary terms when we integrate by part.
\end{proof}

As a key step towards proving the rectifiability of the limit varifold $V$, we now establish an up-to-the-boundary locally uniform density ratio upper bound for the limit measure $\mu$. Note that compared to \cite[Proposition 4.1]{Hutchinson2000}), we do not claim at this stage that there is a uniform lower bound for the density ratio centered at a boundary point. This is due to the possible complication that the $x_i$'s in Lemma \ref{L:Hausdorff} could be a sequence of interior points approaching a boundary point $x \in \spt \mu \cap \partial \Omega$. This poses new difficulties when we want to apply monotonicity formulas which are either centered at a boundary point (Lemma \ref{L:boundary-mono-1}) or stays away from the boundary (Lemma \ref{L:interior-mono-1}).

\begin{proposition}[Density ratio bounds for the limit measure]
\label{P:mu-density-bound}
There exist constants
\[ \theta_1=\theta_1(\Lambda_0, C_0, W, n, i_0, k_0, \kappa_0, \dist_g(\tilde{U},\partial U)) >0, \]
\[ \theta_2=\theta_2(E_0,\Lambda_0, C_0, W, n, i_0, k_0, \kappa_0, \dist_g(\tilde{U},\partial U)) >0 \]
such that for any $x \in (\spt \mu \cap \tilde{U}) \cap \partial \Omega$, we have for all $0 < r < \frac{r_3}{2}$,
\begin{equation}
\label{E:boundary-mu-density-bound}
\frac{\mu(B_r(x))}{\omega_{n-1} r^{n-1}} \leq \theta_2.
\end{equation}
Moreover, for any $x \in (\spt \mu \cap \tilde{U}) \cap \Omega$, we have for all $0<r<\frac{1}{2} \min\{r_3,\dist_g(x,\partial \Omega)\}$, 
\begin{equation}
\label{E:mu-density-bound}
\theta_1 \leq \frac{\mu(B_r(x))}{\omega_{n-1} r^{n-1}} \leq \theta_2.
\end{equation}
Here, $r_3$ is the constant appearing in Proposition \ref{P:density-lower-bound-3} with $b=\frac{1-\alpha}{2}$.
\end{proposition}

\begin{proof}
Let $x \in (\spt \mu \cap \tilde{U}) \cap \partial \Omega$. Fix any $0 < r < \frac{r_3}{2}$. Since $\mu_i \to \mu$ as measures, we will have for all large $i$
\[ \frac{\mu(B_r(x))}{\omega_{n-1} r^{n-1}} \leq \frac{2 \mu_i(B_r(x))}{\omega_{n-1} r^{n-1}} = \frac{2}{h_0} I_{\ep_i,u_i}(r, x).\]
Hence, the desired assertion follows from Proposition \ref{P:density-upper-bound} with $\theta_2:=\frac{2}{h_0} c_9 (E_0+1)$.

For $x \in (\spt \mu \cap \tilde{U}) \cap \Omega$. By Lemma \ref{L:Hausdorff}, after passing to a subsequence, we have a sequence of interior points $x_i \in \Omega \cap \tilde{U}$ such that $|u_i(x_i)| \leq \alpha$ for all $i$ and $x_i \to x$. Fix any $0 < r < \frac{1}{2} \min\{r_3,\dist_g(x,\partial \Omega)\}$. The upper bound $\theta_2$ is obtained exactly as before. For the lower bound, we observe for all large $i$
\[ \frac{\mu(B_r(x))}{\omega_{n-1} r^{n-1}} \geq \frac {\mu_i(B_r(x))}{2 \omega_{n-1} r^{n-1}} \geq \frac{ \mu_i(B_{r/2}(x_i))}{2 \omega_{n-1} r^{n-1}} = \frac{1}{2^n h_0} I_{\ep_i,u_i}(r/2, x_i).\]
Hence, the desired assertion follows from Proposition \ref{P:density-lower-bound-3} with $\theta_1:=\frac{c_8}{2^n h_0} $.
\end{proof}

Exactly the same arguments as in the proof of Lemma \ref{L:Hausdorff} and Proposition \ref{P:mu-density-bound} yields the following result.

\begin{proposition}[Uniform convergence away from limit interface]
\label{P:Hausdorff}
Either $u_i \to 1$ or $u_i \to -1$ locally uniformly on $(\Omega \cap U) \setminus \spt \|V\|$. In particular, we have $M:=\partial^* \{u=1\} \subset \spt \|V\| \cap (\Omega \cap U)$. Moreover, $\frac{\ep_i |\nabla u_i|^2}{2}$ and $\frac{W(u_i)}{\ep_i}$ both converges locally uniformly to $0$ on $(\Omega \cap U) \setminus \spt \|V\|$.
\end{proposition}

\subsection{Vanishing of discrepancy measure}
\label{SS:discrepancy-measure}

Let $e_i:=e_{\ep_i}(u_i)$ and $\xi_i:=\xi_{\ep_i}(u_i)$ denote the energy density and discrepancy function as defined in \eqref{E:energy-density} and \eqref{E:discrepancy}. By compactness of Radon measures (on $U$), after possibly passing to a subsequence, we can assume that there exists a Radon measure $|\xi|$ on $U$ such that
\begin{equation}
\label{E:discrepancy-measure}
|\xi_i| dv_g \mres {\Omega \cap U} \to |\xi| 
\end{equation}
in the sense of measures. We call $|\xi|$ the discrepancy measure. Since $|\xi_i| \leq e_i$, we clearly have $\spt |\xi| \subset \spt \mu \subset \overline{\Omega} \cap U$. 

\begin{theorem}[Vanishing of discrepancy measure]
\label{T:xi-vanish}
The discrepancy measure $|\xi|$ vanishes on $U$, i.e. $|\xi|=0$. In other words, the discrepancy function $\xi_i$ converges to $0$ in $L^1_{loc}(\overline{\Omega} \cap U)$. 
\end{theorem}

\begin{proof}
We first claim that for all $x \in \spt |\xi| \cap \partial \Omega$, we have
\begin{equation}
\label{E:Radon-Nikodym1}
\liminf_{r \to 0} \frac{|\xi|(B_r(x))}{\omega_{n-1}r^{n-1}}=0.
\end{equation}
We argue by contradiction as in \cite[Proposition 4.3]{Hutchinson2000}. Suppose on the contrary that \eqref{E:Radon-Nikodym1} does not hold, i.e. there exist $\delta_0>0$, $R_0>0$ and $x_0 \in \spt |\xi| \cap \partial \Omega$ such that
\begin{equation}
\label{E:vanish-1}
|\xi|(B_r(x_0)) \geq \delta_0 \omega_{n-1} r^{n-1} \quad \text{ for all $0<r\leq R_0$.} 
\end{equation}
By the density ratio upper bound \eqref{E:boundary-mu-density-bound}, there exists $\theta_2>0$ such that
\begin{equation}
\label{E:vanish-2}
\frac{\mu(B_r(x_0))}{\omega_{n-1} r^{n-1}} \leq \theta_2 \quad \text{   for all $0<r\leq \frac{r_3}{2}$.} 
\end{equation}
Fix $0<t < \min\{R_0,\frac{r_3}{2},r_1,r_2,\frac{1}{\kappa_0}\}$ which will be chosen later. Since $|\xi_i| dv_g \mres {\Omega \cap U} \to |\xi|$ and $\mu_i \to \mu$, by \eqref{E:vanish-1} and \eqref{E:vanish-2}, for any fixed $N \in \mathbb{N}$, there exists some $i >N$ such that 
\begin{equation}
\label{E:vanish-4}
\frac{1}{\omega_{n-1} r^{n-1}} \int_{B_r (x_0) \cap \Omega} |\xi_i| \geq \frac{\delta_0}{2} \quad \text{ for all $r \in [t^2,t]$} 
\end{equation}
and 
\begin{equation}
\label{E:vanish-3}
\frac{\mu_i(B_t(x_0))}{\omega_{n-1}t^{n-1}}  \leq 2 \theta_2.
\end{equation}
Note that $|\xi_i|=\xi_i^+ +\xi_i^-$. By \eqref{E:vanish-4} and Proposition \ref{P:xi-integral-bound}, for all $i$ large enough (so that $\ep_i < \ep_7$ and $c_6 \ep_i^{4/5} \leq t^2$) we have for all $\rho \in [t^2,t]$
\begin{equation}
\begin{aligned}
\frac{1}{(1+3 \kappa_0 \rho)\rho^n} \int_{B_\rho(x_0) \cap \Omega} -\xi_i & \geq \frac{1}{4 \rho^n} \int_{B_\rho (x_0) \cap \Omega} |\xi_i| - \frac{2}{\rho^n} \int_{B_\rho(x_0) \cap \Omega} \xi_i^+  \\
& \geq \frac{\delta_0}{8} \rho^{-1} - \frac{2c_7}{\rho^{7/8}} (I_{\ep_i,u_i}(\rho,x_0)+1) .
\end{aligned}
\end{equation}
Since $x_0 \in \spt \mu$, by Proposition \ref{P:density-upper-bound}, for all $i$ large enough (so that $\ep_i <\min\{\ep_9,t^2\}$), we have for all $\rho \in [t^2,t]$
\begin{equation}
\frac{1}{(1+3 \kappa_0 \rho)\rho^n} \int_{B_\rho(x_0) \cap \Omega} -\xi_i  \geq \frac{\delta_0}{8} \rho^{-1} -C(c_7,c_9,E_0) \rho^{-7/8}.
\end{equation}
By the almost monotonicity formula \eqref{E:boundary-mono-1} (note that $\widetilde{\xi}_i \leq \xi_i$), we have for all $\rho \in [t^2,t]$ 
\[ \frac{d}{d\rho}\left( \omega_{n-1} e^{c_1 \rho} \widetilde{I}_{\ep_i,u_i}(\rho,x_0) \right) \geq \frac{\delta_0}{8} e^{c_1 \rho} \rho^{-1} -C(c_7,c_9,E_0) e^{c_1 \rho} \rho^{-7/8}  \]
Integrating over $[t^2,t]$ and using \eqref{E:I-tilde-estimate}, \eqref{E:vanish-3}, we obtain
\[ C(n,c_1,\Lambda_0,C_0) \left( \frac{h_0}{\omega_{n-1}} 2\theta_2 +t \right) \geq \frac{\delta_0}{8} \log \frac{1}{t} - C(c_7,c_9,E_0,c_1) t^{1/8},\] 
which gives a contradiction if $t$ is chosen sufficiently small. Hence, we have proved \eqref{E:Radon-Nikodym1}. 

By the theory of differentiation for Radon measures (applied to $|\xi|$ and $d \mathcal{H}^{n-1} \mres {\partial \Omega}$), c.f. \cite[P.26]{Simon1983}, we conclude from \eqref{E:Radon-Nikodym1} that $|\xi|(\partial \Omega)=0$, i.e. $|\xi|$ vanishes on $\partial \Omega$. By a similar argument using the interior monotonicity formula \eqref{E:interior-mono-1}, we can show that \eqref{E:Radon-Nikodym1} also holds for interior points $x \in \spt |\xi| \cap \Omega$. Combining this with the density lower bound in \eqref{E:mu-density-bound}, we have
\[ \liminf_{r \to 0} \frac{|\xi|(B_r(x))}{\mu(B_r(x))}=0, \]
from which again implies $|\xi|=0$ inside $\Omega \cap U$ as well.
\end{proof}

\subsection{Rectifiability of limit varifold}
\label{SS:rectifiable}

Let $V_i$ be the varifold associated with $u_i$ as in \eqref{E:asso-varifold}. After passing to a subsequence, we can also assume that there exist $u \in BV(\Omega \cap U)$, $\lambda \in \mathbb{R}$ and $V \in \bV_{n-1}(U)$ such that 
\[ \lambda_i \to \lambda, \quad u_i \to u \; a.e., \quad V_i \to V \; \textrm{in the varifold sense.}\]
Our goal is to show that $V$ is in fact rectifiable up to the boundary $\partial \Omega$. Using the vanishing of the discrepancy measure $|\xi|$, we can establish the following result. Recall that $M:=\partial^*\{u=1\}$ is the reduced boundary of $\{u=1\}$ in $\Omega \cap U$ with outward normal $\nu_M$.

\begin{lemma}[First variation of limit varifold]
\label{L:V-1st-variation}
There exist a constant 
\[ c_{11}=c_{11}(E_0,\Lambda_0, C_0, W, n, i_0, k_0, \kappa_0,\dist_g(\tilde{U},\partial U),dv_g(\Omega \cap U), \mathcal{H}^{n-1}(\partial \Omega \cap U)) > 1\]
such that for any $X \in \mathfrak{X}_c(U)$ with $\spt X \subset \tilde{U}$, we have
\begin{equation}
\label{E:V-first-variation-1}
\left| h_0 \delta V(X) +\lambda \int_{\Omega \cap U} X \cdot Du \right| \leq c_{11} \sup_{\partial \Omega \cap U} |X \cdot \nu|.
\end{equation}
In particular, for any $X \in \mathfrak{X}^{tan}_c(U)$, we have
\begin{equation}
\label{E:V-first-variation-2}
\delta V(X)= -\frac{2\lambda}{h_0} \int_M X \cdot \nu_M \; d\mathcal{H}^{n-1}.
\end{equation}

\end{lemma}

\begin{remark}
If $\lambda=0$, then \eqref{E:V-first-variation-2} implies that $V$ is a free boundary stationary varifold (c.f. \cite[Definition 3.1]{Guang2020}).
\end{remark}

\begin{proof}
Let $X \in \mathfrak{X}_c(U)$ with $\spt X \subset \tilde{U}$. Note that $\delta V_i(X) \to \delta V(X)$ as $i \to \infty$. On the other hand, by the calculations in \cite[Lemma 4.1]{Kagaya2018}, we have
\begin{equation}
\label{E:rectifiable-1}
\begin{aligned}
h_0 \delta V_i(X) = & \int_{\Omega \cap \{|\nabla u_i| \neq 0\}} \nabla X \cdot \left( \frac{\nabla u_i}{|\nabla u_i|} \otimes \frac{\nabla u_i}{|\nabla u_i|} \right) \xi_i  - \int_{\Omega \cap \{|\nabla u_i| =0\}} (\nabla X \cdot I) \frac{W(u_i)}{\ep_i}   \\
&+ \lambda_i \int_{\Omega} u_i \Div_g X + \int_{\partial \Omega} \left( \frac{\ep_i |\nabla u_i|^2}{2} + \frac{W(u_i)}{\ep_i} -\lambda_i u_i \right) (X \cdot \nu) d \mathcal{H}^{n-1}.  
\end{aligned}
\end{equation}
Note that the first and second term of the (R.H.S.) above are both bounded by  $C(n)\|X\|_{C^1} \int_{\tilde{U} \cap \Omega} |\xi_i|$, which converges to $0$ as $i \to \infty$ by Theorem \ref{T:xi-vanish}. For the third term, since $\lambda_i \to \lambda$ and $u_i \to u \in BV(\Omega \cap U)$ a.e., by Dominated Convergence theorem we have
\[ \lim_{i \to \infty} \lambda_i \int_{\Omega} u_i \Div_g X = \lambda \int_{\Omega} u \Div_g X =  - \lambda \int_{\Omega} X \cdot D u  + \lambda \int_{\partial \Omega} Tu (X \cdot \nu) d\mathcal{H}^{n-1}\]
where $Tu \in L^1(\partial \Omega)$ is the trace of $u$ on $\partial \Omega$ (c.f. \cite[Theorem 5.6]{Evans2015}). Moreover, by the boundednesss of the trace operator $T$, we have
\[ \|Tu\|_{L^1(\partial \Omega \cap U)} \leq C(n,i_0,k_0,\kappa_0) \|u\|_{BV(\Omega \cap U)} \leq C(n,i_0,k_0,\kappa_0, E_0,W, dv_g(\Omega \cap U)).\]
Combining all the above and applying Lemma \ref{L:local-boundary-energy} for the last term gives \eqref{E:V-first-variation-1}. 
\end{proof}

\begin{theorem}[Rectifiability of limit varifold]
\label{T:V-rectifiable}
The limit varifold $V \in \bV_{n-1}(U)$ satisfies $\|V\|=\mu$ and is rectifiable. 
\end{theorem}

\begin{proof}
Recall that $\|V_i\|=\mu_i \mres {\{|\nabla u_i| \neq 0\}}$ and since $h_0 \mu_i \mres {\{ |\nabla u_i|=0\}} \leq |\xi_i| dv_g \mres {\Omega \cap U}$ which converges to $0$ by Theorem \ref{T:xi-vanish}, we have $\|V_i\| \to \mu$ and thus $\|V\|=\mu$. From \eqref{E:V-first-variation-1}, we know that $V$ has locally bounded first variation in $U$. By Allard's Rectifiability Theorem \cite[Theorem 42.4]{Simon1983}, it remains to show that $\|V\|=\mu$ has a locally uniform density lower bound on $\spt \|V\|$.

We claim that $V$ has free boundary in $\partial \Omega \cap U$ as described in Section \ref{SS:varifolds}. Recall $M=\partial^* \{u=1\} \subset \Omega \cap U$. By the structure theorem for sets of finite perimeter (c.f. \cite[Theorem 5.15]{Evans2015}, $\|\partial \{u=1\}\| = \mathcal{H}^{n-1} \mres M$ which is a Radon measure on $U$ by \eqref{E:u-BV}. Consider the Radon measures $\mu=\|V\|$ and $\sigma:=\mathcal{H}^{n-1} \mres M$ on $U$, the density ratio lower bound \eqref{E:mu-density-bound} implies (c.f. \cite[Section 1.6]{Evans2015}) that $\sigma$ is absolutely continuous with respect to $\mu$. Together with Lemma \ref{L:V-1st-variation}, the Radon-Nikodym Theorem implies that
\[ \delta V(X)= -\frac{2\lambda}{h_0} \int_{\Omega \cap U} X \cdot \nu_M d\sigma = -\frac{2 \lambda}{h_0} \int_{\Omega \cap U} (X \cdot \nu_M) D_\mu \sigma d\mu \]
for any $X \in \mathfrak{X}_c^{tan}(U)$. Therefore, $V$ has free boundary in $\partial \Omega \cap U$ with $\|V\|$-measurable vector field 
\[ H^T(x):=\left\{ \begin{array}{cl}
\frac{2\lambda}{h_0} (D_\mu \sigma) \nu_M & \text{ for $x \in M$,}\\
0 & \text{ otherwise.} 
\end{array} \right.  \]
Note the the density lower bound in Proposition \ref{P:mu-density-bound} implies that $H^T \in L^\infty$.

To finish the proof, note that $\|V\|$ has a locally uniform density lower bound on $\spt \|V\| \cap \Omega \cap U$ by \eqref{E:mu-density-bound}. As $V$ has free boundary in $\partial \Omega \cap U$ with $H^T \in L^\infty$, the ``reflected'' density of $V$ is upper semi-continuous (c.f. \cite[Corollary 3.2]{Grueter1986}) and hence $\|V\|$ has a locally uniform density lower bound up to the boundary $\partial \Omega$. This completes the proof.
\end{proof}

We now give the proof of our main result, Theorem \ref{T:main}, modulo the integrality of the limit varifold $V$ which will be presented in the next section.

\begin{proof}[Proof of Theorem \ref{T:main}]
(1) follows directly from Theorem \ref{T:xi-vanish} and \ref{T:V-rectifiable}. (2) follows from Lemma \ref{L:Hausdorff}, Proposition \ref{P:mu-density-bound} and \ref{P:Hausdorff}. The proof of (3) will be given in Section \ref{S:integrality}.Finally, (4) follows from the proof of Theorem \ref{T:V-rectifiable} and that $\Theta(x)=(D_\mu \sigma(x))^{-1}$. The rest then follows from \cite[Proposition 3.2]{Edelen2020}.
\end{proof}

\section{Integrality of the limit varifold}
\label{S:integrality}

We will establish in this section that the limit varifold $V$ is integral in $U$, i.e. $V \in \bIV_{n-1}(U)$. Since $V$ is rectifiable in $U$ by Theorem \ref{T:V-rectifiable}, $V$ has a unique approximate tangent space $T$ at $\mathcal{H}^{n-1}$-a.e. $x \in \spt \|V\|$. In other words, for any such $x \in \spt \|V\|$ and any sequence of positive real numbers $r_i \to 0$, we have 
\[ (\Phi_{r_i})_\sharp V \to \Theta(x) |T|\]
as varifolds \footnote{To be more precise, one can first locally isometrically embed $(\tilde{M},g)$ into some Euclidean space $\mathbb{R}^L$ and then do rescalings in $\mathbb{R}^L$ (c.f. \cite{Allard1972}). Alternatively, one can also do the rescalings intrinsically using the exponential map (c.f. \cite{Scharrer2022}). Both are equivalent and thus we will be using either formulation.} Here, $(\Phi_{r_i})_\sharp V$ is the pushforward of the varifold $V$ by the map $\Phi_{r_i}:(\widetilde{M},g) \to (\widetilde{M}, r_i^{-2} g)$, $\Theta(x)$ is the $(n-1)$-dimensional density of $V$ at $x$ and $|T|$ is the unit density varifold associated to the hyperplane $T \subset T_x \widetilde{M} \cong \mathbb{R}^n$.

We shall assume $x \in \partial \Omega \cap U$ as the arguments for interior points are similar (see also \cite{Hutchinson2000,Tonegawa2002,Tonegawa2005} in the Euclidean setting). Fix any $r_i \to 0$. We use $\tilde{g}_i:=r_i^{-2} g$ to denote the rescaled metric on $\widetilde{M}$. Note that $(\overline{\Omega} \cap U,x,\tilde{g}_i) \to (\mathbb{R}^n_+,0,g_E)$ smoothly in the sense of pointed Cheeger-Gromov convergence. Here $\mathbb{R}^n_+:=\{(x_1,\cdots,x_n) \in \mathbb{R}^n \; : \; x_n \geq 0 \}$ is the upper half-space equipped with the standard Euclidean metric $g_E$. Since $V_i \to V$, after passing to a subsequence, we can further assume as $i \to \infty$
\begin{equation}
\label{E:intgeral-1}
(\Phi_{r_i})_\sharp V_i \to \Theta(x) |T| \quad \text{ and } \quad \ep_i r_i^{-1} \to 0.
\end{equation}
Let $\tilde{V}_i=(\Phi_{r_i})_\sharp V_i$, $\tilde{\ep}_i= \ep_i r_i^{-1}$ and $\tilde{\lambda}_i=r_i \lambda_i$. Then $u_i$ satisfies, with respect to the rescaled metric $\tilde{g}_i$,
 \begin{equation}
\label{E:1a-II}
-\tilde{\ep}_i \Delta_{\tilde{g}_i} u_i + \frac{W'(u_i)}{\tilde{\ep}_i}= \tilde{\lambda}_i \quad  \text{in $\Omega \cap U$},
\end{equation}
\begin{equation}
\label{E:1b-II}
\frac{\partial u_i}{\partial \nu}=0 \quad \text{on $\partial \Omega \cap U$}.
\end{equation}
Note that we have $\|u_i\|_{L^\infty} \leq C_0$, $|\tilde{\lambda}_i| \leq r_i \Lambda_0 \to 0$. Moreover, by Proposition \ref{P:density-upper-bound}, for any fixed $R>0$, we have for all sufficiently large $i$ 
\[ (E_{\tilde{\ep}_i} \mres (B_R(x),\tilde{g}_i))(u_i) = \omega_{n-1} R^{n-1} I_{\ep_i,u_i}(r_i R,x) \leq \omega_{n-1} R^{n-1}c_9 (E_0+1). \]
Therefore, all the results obtained in previous sections can be applied (with slight modifications for a convergent sequence of varying metrics). Observe that $\tilde{V}_i$ is the associated varifold of $u_i$ with respect to the rescaled metric $\tilde{g}_i$ and $\tilde{V}_i \to \tilde{V}:=\Theta(x) |T|$. Since $\tilde{\lambda}_i \to 0$, by Lemma \ref{L:V-1st-variation}, $\tilde{V}_i \to \tilde{V}:=\Theta(x) |T|$ is a stationary varifold with free boundary in $\mathbb{R}^{n-1}:=\partial \mathbb{R}^n_+$. As $\tilde{V}$ is supported in $\mathbb{R}^n_+$ and $T$ is a hyperplane, we must have $T=\{x_n=0\}$.

For our convenience, by a suitable choice of coordinates, we can assume $\overline{\Omega} = \mathbb{R}^n_+$ is equipped with a metric $\tilde{g}_i$ converging locally smoothly to the Euclidean metric $g_E$ and $U=B_4(0)$ is the Euclidean ball of radius $4$ centered at the origin. All the balls and distances will be taken with respect to $g_E$ from now on. As in Section \ref{S:monotonicity}, we denote throughout this section $u_i$, $\tilde{\lambda}_i$, $\tilde{\ep}_i$, $\tilde{g}_i$ by $u$, $\tilde{\lambda}$, $\tilde{\ep}$, $\tilde{g}$. We use $T:\mathbb{R}^n \to \mathbb{R}^{n-1}$ to denote the orthogonal projection $T(x_1,\cdots,x_n)=(x_1,\cdots,x_{n-1})$ with respect to $g_E$ and $T^\perp:\mathbb{R}^n \to \mathbb{R}$ to denote the map $T^\perp (x_1,\cdots,x_n)=x_n$. We also define $\nu=(\nu_1,\cdots,\nu_n)$ to be the unit normal to the level sets of $u$ as
\[ \nu:=\left\{ \begin{array}{cl}
\frac{\nabla u}{|\nabla u|} & \text{ when $|\nabla u| \neq 0$,} \\
0 & \text{ when $|\nabla u|=0$.}
\end{array} \right. \]
We also denote
\[ e_{\tilde{\ep}}:=\frac{\tilde{\ep} |\nabla u|^2}{2} + \frac{W(u)}{\tilde{\ep}}, \quad \xi_{\tilde{\ep}}:=\frac{\tilde{\ep} |\nabla u|^2}{2} - \frac{W(u)}{\tilde{\ep}}.\]

We first establish a proposition which shows that the smallness of the discrepancy and tilt excess imply the solution $u$ is close in $\tilde{\ep}$-scale to the one-dimensional heteroclinic solution near the transition region which must be rather far away from the boundary.

\begin{proposition}[Sheets separation]
\label{P:sheet-separation}
For any given $s, b \in (0,1)$, there exist constants
\[ \eta_1=\eta_1(s,b,W) \in (0,1),\]
\[ L=L(s,b,W)>1\] 
such that the following holds: for any $0 < \eta < \eta_1$, suppose we have
\begin{itemize}
\item[(i)] $y \in \mathbb{R}^n_+$ with $T(y)=0$ and $|u(y)| \leq 1-b$;
\item[(ii)]  $u$ satisfies \eqref{E:1a-II} on $B_{4\tilde{\ep} L}(y) \cap \mathbb{R}^n_+$ and \eqref{E:1b-II} along $B_{4\tilde{\ep} L}(y) \cap \partial \mathbb{R}^n_+$ with $\tilde{\ep} \in (0,1)$, $|\tilde{\lambda}| \leq \eta$, $\|\tilde{g} - g_E\|_{C^2(B_{4\tilde{\ep}L}(y))} \leq \eta$ and $|\xi_{\tilde{\ep}}| \leq \eta \tilde{\ep}^{-1}$ on $B_{4\tilde{\ep} L}(y) \cap \mathbb{R}^n_+$;
\item[(iii)] 
\[ \int_{B_{4\tilde{\ep} L}(y) \cap \mathbb{R}^n_+} |\xi_{\tilde{\ep}}| + (1-(\nu_n)^2) \tilde{\ep} |\nabla u|^2 \leq \eta (4\tilde{\ep} L)^{n-1},\]
\end{itemize}
then $B_{3\tilde{\ep}L}(y) \cap \partial \mathbb{R}^n_+=\emptyset$, $T^{-1}(0) \cap \{ x \in B_{3 \tilde{\ep} L}(y) :  u(x)=u(y)\}=\{y\}$ and
\[ |I_{\tilde{\ep},u}(\tilde{\ep}L,y) -h_0 | \leq s. \]
\end{proposition}

\begin{proof}
Exactly the same proof (the weaker hypothesis $|\xi_{\tilde{\ep}}| \leq \eta \tilde{\ep}^{-1}$ is sufficient as pointed out in \cite[Proposition 4.8]{Tonegawa2002}) of \cite[Proposition 5.6]{Hutchinson2000} goes through using the up-to-the-boundary $C^2$-estimates in \cite{Agmon1959} and shows that $u$ is in $\tilde{\ep}$-scale $C^1$-close to the unique one-dimensional finite-energy heteroclinic solution $q(x_n)$ satisfying $q(|y|)=u(y)$. In particular, this implies $\frac{\partial u}{\partial x_n} \neq 0$ on $B_{3 \tilde{\ep}L}(y) \cap \mathbb{R}^n_+$ from which the assertions follow by \eqref{E:1b-II}.
\end{proof}

The second proposition deals with decomposing the limit varifold $V$ into stacked single-layered interfaces. Note that by Proposition \ref{P:sheet-separation}, we only need to consider the case where the transition regions stay sufficiently away from the boundary and hence the interior arguments in \cite{Hutchinson2000,Tonegawa2002,Tonegawa2005} go through with trivial modifications.

\begin{proposition}[Stacking argument]
\label{P:stacking}
For any given $R,\tilde{E}_0>0$, $s \in (0,1)$, $N \in \mathbb{N}$, there exists a constant
\[ \eta_2=\eta_2(R, C_0, \tilde{E}_0,s,N,n,W)>0\]
such that the following holds: for any $0 < \eta < \eta_2$, suppose we have
\begin{itemize}
\item[(i)] $Y \subset \mathbb{R}^n_+$ has no more than $N+1$ elements, $T(y)=0$ for all $y \in Y$, diam $Y \leq \eta R$ and there exists $a>0$ such that $|y-z| >3a$ for all $y,z \in Y$; moreover, $U_{3a}(Y) \cap \partial \mathbb{R}^n_+ = \emptyset$;
\item[(ii)] $u$ satisfies \eqref{E:1a-II} on $U_R(Y) \cap \mathbb{R}^n_+$ and \eqref{E:1b-II} along $U_R(Y) \cap \partial \mathbb{R}^n_+$ with $\|u\|_{L^\infty} \leq C_0$, $\tilde{\ep} <\eta$, $\|\tilde{g} - g_E\|_{C^2(U_R(Y))} \leq \eta$ and $|\tilde{\lambda}| \leq \eta$; moreover, for each $y \in Y$, we have
\[ \int_a^R \frac{d\tau}{\tau^n} \int_{B_\tau(y) \cap \mathbb{R}^n_+} \xi^+_{\tilde{\ep}} \leq \eta R;\]
\item[(iii)] for each $y \in Y$ and $r \in [a,R]$, we have 
\[ \int_{B_r(y) \cap \mathbb{R}^n_+} |\xi_{\tilde{\ep}}| + (1-(\nu_n)^2) \tilde{\ep} |\nabla u|^2 \leq \eta r^{n-1}, \]
\[ \int_{B_r(y) \cap \mathbb{R}^n_+} \tilde{\ep} |\nabla u|^2 \leq \tilde{E}_0 r^{n-1};\]
\end{itemize}
then we must have
\[ \sum_{y \in Y} \frac{1}{\omega_{n-1}a^{n-1}} \int_{B_a(y)} e_{\tilde{\ep}} \leq s+ \frac{1+s}{\omega_{n-1}R^{n-1}} \int_{U_R(Y)} e_{\tilde{\ep}}.\]
\end{proposition} 

\begin{proof}
The proof of Proposition \ref{P:stacking} requires a technical lemma similar to \cite[Lemma 5.4]{Hutchinson2000} except that we have a weaker hypothesis on the upper bound of the discrepancy function (c.f. \cite[Proposition 4.7]{Tonegawa2002}). We then argue exactly the same way by plugging in the radial vector field centered at a point in $Y$ which is cutoff radially and along the $x_n$-direction (c.f. the proof of \cite[Lemma 5.4]{Hutchinson2000}) and control the error terms in the monotonicity formula as in Lemma \ref{L:density-lower-bound-1} and Proposition \ref{P:density-lower-bound-3}. 
\end{proof}

The third proposition says that the energy is uniformly small in $\tilde{\ep}$ outside the transition region (c.f. \cite[Proposition 5.1]{Hutchinson2000} and \cite[Proposition 4.5]{Tonegawa2005}).

\begin{proposition}[Uniform smallness of energy outside transition region]
\label{P:small-energy-estimate}
Suppose that $\|\tilde{g} - g_E\|_{C^2(B_{4}(0) \cap \mathbb{R}^n_+)} \leq 1$ and $|\tilde{\lambda}|\leq 1$. Given $s>0$, there exist positive constants 
\[ \ep_{12}=\ep_{12}(s, C_0, n, W)>0,\]
\[ b=b(s, C_0, n, W, E_0) \in (0,1)\]
such that for any $0 <\tilde{\ep} < \epsilon_{12}$, we have 
\[ \int_{B_3(0) \cap \mathbb{R}^n_+ \cap \{|u| \geq 1-b\}} \frac{W(u)}{\tilde{\ep}} \leq s.\]
\end{proposition}

The proof of Proposition \ref{P:small-energy-estimate} requires two lemma as in \cite{Hutchinson2000}. Define
\begin{equation}
\label{E:Z_alpha}
Z_\alpha := \left\{ x \in B_4(0) \cap \mathbb{R}^n_+ \; : \; |u(x)| \leq  \frac{1+\alpha}{2}  \right\}.
\end{equation}

\begin{lemma}
\label{L:small-energy-estimate-1}
Suppose that $\|\tilde{g} - g_E\|_{C^2(B_{4}(0) \cap \mathbb{R}^n_+)} \leq 1$ and $|\tilde{\lambda}|\leq 1$. There exist positive constants
\[ \ep_{13}=\ep_{13}(W)>0,\]
\[ c_{13}=c_{13}(n,W)>0,\]
such that for any $0<\tilde{\ep}<\ep_{13}$, $x \in B_3(0) \cap \mathbb{R}^n_+$ with $|u(x)| <1-\tilde{\ep}^\beta$ for some $\beta$ satisfying
\[ \frac{1}{c_{13} |\log \tilde{\ep}|} < \beta < \min \left\{ \frac{2}{3}, \frac{1}{c_{13} \tilde{\ep} |\log \tilde{\ep}|} \right\}, \]
then we have 
\[ \dist (x,Z_\alpha) \leq c_{13} \beta \tilde{\ep} |\log \tilde{\ep}|.\]
\end{lemma}

\begin{proof}
We argue exactly the same way as in \cite[Lemma 5.2]{Hutchinson2000} except that we have to modify the function $\psi$ suitably for our problem. By a translation, we assume $x=0$ instead and denote $\tilde{\Omega}=\tilde{\ep}^{-1} (\Omega-x)$. By Fredholm alternative, for any $0 <R < \tilde{\ep}^{-1}$, there exists a unique solution $\varphi=\varphi_R$ to the following mixed boundary value problem:
\begin{equation}
\left\{ \begin{array}{cl}
\Delta \varphi = \frac{\kappa}{4} \varphi & \text{on $B_R(0) \cap \tilde{\Omega}$}, \\
\frac{\partial \varphi}{\partial \nu} = 0 & \text{along $B_R(0) \cap \partial \tilde{\Omega}$},\\
\varphi=1 & \text{along $\partial B_R(0) \cap  \tilde{\Omega}$}.
\end{array} \right. 
\end{equation}
Note that $\varphi_R$ is continuous on $\overline{B_R(0) \cap \tilde{\Omega}}$ up to the corner $\partial B_R(0) \cap \partial \tilde{\Omega}$. Moreover, $\varphi_R>0$ by the strong maximum principle and unique continuation property (see e.g. \cite[Section 1.8]{Lieberman2013}). We can then define $\psi_R=(\varphi_R(0))^{-1} \varphi_R$ which satisfies
\begin{equation}
\left\{ \begin{array}{cl}
\Delta \psi_R = \frac{\kappa}{4} \psi_R & \text{on $B_R(0) \cap\tilde{\Omega}$}, \\
\frac{\partial \psi_R}{\partial \nu} = 0 & \text{along $B_R(0) \cap \partial \tilde{\Omega}$},\\
\psi_R(0)=1. & 
\end{array} \right. 
\end{equation}
Using Harnack inequality and Schauder estimates \cite{Agmon1959}, there exists a constant $c_{13}>0$ depending only on $n$ and $\kappa$ such that $\psi_R(x)>e^{|x|/c_{13}}$ for all $|x| \geq 1$. Using $\psi_R$ in place of $\psi$, the rest of the proof follows from the arguments in \cite[Lemma 5.2]{Hutchinson2000}.
\end{proof}

\begin{lemma}
\label{L:small-energy-estimate-2}
Suppose that $\|\tilde{g} - g_E\|_{C^2(B_{4}(0) \cap \mathbb{R}^n_+)} \leq 1$ and $|\tilde{\lambda}|\leq 1$. There exist positive constants 
\[ \ep_{14}=\ep_{14}(C_0,W,n)>0,\]
\[ c_{14}=c_{14}(C_0,W,n,E_0)>0,\]
such that for any $0 < \tilde{\ep} < \ep_{14}$ and $\tilde{\ep} < r < r_3$, we have 
\[ \mathcal{L}^n\left(\{x \in B_3(0) \cap \mathbb{R}^n_+ \; : \; \dist (x,Z_\alpha) < r\}   \right) \leq c_{14} r.\]
\end{lemma}

\begin{proof}
The proof follows the same covering arguments from \cite[Lemma 5.3]{Hutchinson2000} using Proposition \ref{P:density-lower-bound-3} (with $b=\frac{1-\alpha}{2}$) and the arguments in the proof of Proposition \ref{P:xi-integral-bound}.
\end{proof}

\begin{proof}[Proof of Proposition \ref{P:small-energy-estimate}]
It follows exactly the same proof in \cite[Proposition 5.1]{Hutchinson2000} using Lemma \ref{L:small-energy-estimate-1} and \ref{L:small-energy-estimate-2}. 
\end{proof}

\begin{proof}[Proof of Theorem \ref{T:main} (continued)]
We now finish the proof of Theorem \ref{T:main} by showing that $V \in \bIV_{n-1}(U)$. We use the same notations as in the beginning of this section with $x=0 \in \partial \Omega \cap U$. Recall that 
\begin{equation}
\label{E:intgeral-2}
\tilde{V}_i:=(\Phi_{r_i})_\sharp V_i \to \Theta_0 |T| \quad \text{ and } \quad \tilde{\ep_i}:=\ep_i r_i^{-1} \to 0.
\end{equation}
Our goal here is to prove that $\Theta_0 \in \mathbb{N}$.

First of all, by Theorem \ref{T:xi-vanish} and \eqref{E:intgeral-2}, we have
\begin{equation}
\label{E:tilt-excess-estimate}
\lim_{i \to \infty} \int_{B_4(0) \cap \mathbb{R}^n_+}  |\xi_{\tilde{\ep}_i}| = \lim_{i \to \infty} \int_{B_4(0) \cap \mathbb{R}^n_+} (1-\nu_n^2) \tilde{\ep_i} |\nabla u_i|^2 =0.
\end{equation}
Let $N \in \mathbb{N}$ be the smallest positive integer such that $N>\Theta_0$. Note that $\|\tilde{g}_i - g_E\|_{C^2(B_R(0) \cap \mathbb{R}^n_+)} \to 0$ and $|\tilde{\lambda}_i| \to 0$ as $i \to \infty$. Fix an arbitrary small $s>0$, by Proposition \ref{P:small-energy-estimate} and \eqref{E:tilt-excess-estimate}, we can choose some $b>0$ such that for all sufficiently large $i$, we have
\begin{equation}
\label{E:energy-estimate} 
\int_{B_3(0) \cap \mathbb{R}^n_+ \cap \{|u_i| \geq 1-b\}} \left( \frac{\tilde{\ep}_i |\nabla u_i|^2}{2} + \frac{W(u_i)}{\tilde{\ep}_i} \right) \leq s.
\end{equation}
With these choices of $s$ and $b$, we choose $0<\eta < \min\{\eta_1,\eta_2\}$ and $L>1$ as in Proposition \ref{P:sheet-separation} and \ref{P:stacking} (with $R=1$ and $\tilde{E}_0=c_9 (E_0+1)$). We then define as in \cite{Hutchinson2000} the good set $G_i$ consisting of all $x \in B_2(0) \cap \mathbb{R}^n_+ \cap \{|u_i| \leq 1-b\}$ such that
\[ \int_{B_r(x) \cap \mathbb{R}^n_+} |\xi_{\tilde{\ep}_i}| + (1-\nu_n^2) \tilde{\ep_i} |\nabla u_i|^2 \leq \eta r^{n-1} \quad \textrm{for all $r \in [4 \tilde{\ep}_i L, 1]$.}\]
By Besicovitch covering theorem and \eqref{E:tilt-excess-estimate}, we have
\[ \lim_{i \to \infty} \|\tilde{V}_i\|(B_2(0) \cap \mathbb{R}^n_+ \cap \{|u_i| \leq 1-b\} \setminus G_i) =0 \]
and 
\[ \lim_{i \to \infty} \mathcal{H}^{n-1}(T(B_2(0) \cap \mathbb{R}^n_+ \cap \{|u_i| \leq 1-b\} \setminus G_i))=0.\]
On the other hand, by \eqref{E:intgeral-2} and Proposition \ref{P:density-lower-bound-3}, we obtain
\begin{equation}
\label{E:Hausdorff-convergence}
\lim_{i \to \infty} \dist(G_i,T)=0.
\end{equation}

Now, for any $x \in B_1 (0) \cap \partial \mathbb{R}^n_+$ and $|t| \leq 1-b$, we let 
\[Y:=T^{-1}(x) \cap G_i \cap \{u_i=t\} \qquad \textrm{ and } \qquad a:=L\tilde{\ep}_i.\]
We would like to apply Proposition \ref{P:stacking} to this $Y$ for all sufficiently large $i$. Note that $\|\tilde{g}_i - g_E\| \to 0$ and $\tilde{\lambda}_i \to 0$ as $i \to \infty$. By Proposition \ref{P:sheet-separation} and \eqref{E:Hausdorff-convergence}, all assumptions in (i) of Proposition \ref{P:stacking} are satisfied (except perhaps the cardinality of $Y$) for large $i$. Noting how the constants in Section \ref{S:monotonicity} scale with respect to the blow-up radii $r_i \to 0$ (especially Proposition \ref{P:xi-bound} and \ref{P:xi-integral-bound}), (ii) is also satisfied for large $i$. Condition (iii) is essentially the definition of $G_i$, together with Proposition \ref{P:density-upper-bound}. Note that $Y$ cannot contain more than $N-1$ elements since otherwise we will obtain by Proposition \ref{P:sheet-separation} and \ref{P:stacking} that
\[ N \leq (N+1)s + (1+s)(\Theta_0+s) \]
which is a contradiction when $s$ is sufficiently small (depending only on $N$). The rest of the argument is the same as \cite[Theorem 1]{Hutchinson2000} and we must have $\Theta_0=N-1 \in \mathbb{N}$.
\end{proof}

\bibliographystyle{amsplain}
\bibliography{references}

\end{document}